\documentclass[a4paper,reqno,11pt]{amsart}
\usepackage{graphicx}
\usepackage{amsmath,amsxtra,amssymb,latexsym, amscd,amsthm}
\usepackage[mathscr]{eucal}
\usepackage{color}
\usepackage{enumerate}
\usepackage[colorlinks,pdfdisplaydoctitle,linkcolor=blue,citecolor=red]{hyperref}
\theoremstyle{plain}
\newtheorem{theorem}{Theorem}[section]
\newtheorem{assum}{Assumption}[section]
\newtheorem{rem}{Remark}[section]
\newtheorem{lemma}{Lemma}[section]

\newtheorem{corollary}{Corollary}[section]
\theoremstyle{definition}

\numberwithin{equation}{section}

	\def\E{\mathbb{E}}
	\def\M{\mathrm{M}}

\begin{document}
\title[Rate of convergence for mean-field SDEs]{Rate of convergence  in the Smoluchowski-Kramers approximation for mean-field stochastic differential equations}

\begin{abstract}
In this paper we study a second-order mean-field stochastic differential systems describing the movement of a particle under the influence of a time-dependent force, a friction, a mean-field interaction and a space and time-dependent stochastic noise. Using techniques from Malliavin calculus, we establish explicit rates of convergence in the zero-mass limit (Smoluchowski-Kramers approximation) in the $L^p$-distances and in the total variation distance for the position process, the velocity process and a re-scaled velocity process to their corresponding limiting processes.  
\end{abstract}
		
\author[T. C. Son]{Ta Cong Son}
\address[T. C. Son]{University of Science, Vietnam National University, Hanoi}
\email{congson82@gmail.com}

\author[D. Q. Le]{Dung Quang Le}
\address[D. Q. Le]{École Polytechnique, France}
\email{quangdung0110@gmail.com}

\author[M. H. Duong]{Manh Hong Duong}
\address[M. H. Duong]{University of Birmingham, UK.} 
\email{H.Duong@bham.ac.uk.}

\maketitle

	\section{Introduction}
	
In this paper, we are interested in the following second-order mean-field stochastic differential equations
\begin{equation}\label{eq1}
		\begin{cases}dX^\alpha_t=Y^\alpha_t\, dt,\\
			\frac{1}{\alpha}dY^\alpha_t=[-\kappa Y^\alpha_t-g(t,X^\alpha_t)-\gamma(Y^\alpha_t-\E(Y^\alpha_t))]\,dt+\sigma(t,X^\alpha_t)\,dW_t,\\
			X^\alpha_0=x_0,Y^\alpha_0=y_0.
		\end{cases}
	\end{equation}
Here $\alpha, \gamma$ and $\kappa$ are positive constants, $g(t,x):[0,T]\rightarrow\mathbb{R}$ is a given function, $x_0,y_0 \in\mathbb{R}$ are given points in the real line, and $(W_t)_{t\geq 0}$ is the standard one-dimensional Wiener process. The notation $\E$ denotes the expectation with respect to the probability measure of the underlying probability space in which the Wiener process is defined.

System \eqref{eq1} describes the movement of a particle at position (displacement) $X_t^\alpha\in\mathbb{R}$ and with velocity $Y_t^\alpha\in\mathbb{R}$, at time $t$, under the influence of four different forces: an external, possibly time-dependent and non-potential, force  $-g(t,X^\alpha_t)$; a friction $-\kappa Y^\alpha_t$; a (McKean-Vlasov type) mean-field interaction force $-\gamma(Y^\alpha_t-\E(Y^\alpha_t))$ (noting that here the mean-field term is acting on the velocity rather than the position) and a stochastic noise $\sigma(t,X^\alpha_t)\dot{W}_t$. Physically, $\alpha$ is the inverse of the mass, $\kappa$ is the friction coefficient and $\gamma$ is the strength of the interaction.  We use the superscript $\alpha$ in \eqref{eq1} to emphasize the dependence on $\alpha$ since in the subsequent analysis we are concerned with the asymptotic behaviour of \eqref{eq1} as $\alpha$ tends to $+\infty$.

Under Assumptions \ref{gt:2.1} (see below) of this paper, system \eqref{eq1} can also be obtained as the mean-field (hydrodynamic) limit of the following interacting particle system as $N$ tends to $+\infty$
 \begin{equation}\label{eq1n}
		\begin{cases}dX^{\alpha,i}_t=Y^{\alpha,i}_t\, dt,\\
			dY^{\alpha,i}_t=[-\alpha\kappa Y^{\alpha,i}_t-\alpha g(t,X^{\alpha,i}_t)-\frac{\alpha\gamma}{N}\sum_{j=1}^N(Y^{\alpha,i}_t-Y^{\alpha,j}_t)]\,dt+\alpha \sigma(t,X^{\alpha,i}_t)\,dW^i_t,\\
			X^\alpha_0=x_0,Y^\alpha_0=y_0,
		\end{cases}
	\end{equation}
where $\{W^i\}_{i=1}^N$ are independent one-dimensional Wiener processes. In fact, under Assumptions \ref{gt:2.1} the above interacting system satisfies the property of propagation of chaos, that is as $N$ tends to infinity, it behaves more and more like a system of independent particles, in which each particle evolves according to \eqref{eq1} where the interaction term in \eqref{eq1n} is replaced by the expectation one. For a detailed account on the propagation of chaos phenomenon, we refer the reader to classical papers \cite{kac1956,Sznitman1991} and more recent papers~\cite{BolleyGuillinMalrieu2010,Duong2015NA,Jabin2017} and references therein for degenerate diffusion systems like \eqref{eq1}. The interacting particle system \eqref{eq1n} and its mean-field limit \eqref{eq1} and more broadly systems of these types have been used extensively in biology, chemistry and statistical physics for the modelling of molecular dynamics, chemical reactions, flockings, social interactions, just to name a few, see for instance, the monographs~\cite{risken1996fokker,pavliotis2014stochastic}.
 
In this paper, we are interested in the zero-mass limit (as also known as the Smoluchowski-Kramers approximation) of \eqref{eq1}, that is its asymptotic behaviour as $\alpha$ tends to $+\infty$. By employing techniques from Malliavin calculus, we obtain explicitly rate of convergences, in $L^p$-distances and in total variation distances, for both the position and velocity processes. 
\subsection{Main results}
Before stating our main results, we make the following assumptions.
\begin{assum}\label{gt:2.1}
\begin{itemize}\
			\item[$(A)$] The coefficients $g,\sigma:[0,T]\times \mathbb{R}\longrightarrow \mathbb{R}$ have linear growth, i.e. there exists $K>0$ such that
			$$|g(t,x)|+|\sigma(t,x)|\leq K (1+|x|)\quad \forall x\in \mathbb{R}, t\in[0,T].$$
			\item[$(B)$]  The coefficients $g,\sigma:[0,T]\times \mathbb{R}\longrightarrow \mathbb{R}$ are Lipschitz, i.e. there exists $L>0$ such that
			$$ |g(t,x)-g(t,y)|+|\sigma(t,x)-\sigma(t,y)| \leq L|x-y| \quad \forall x, y\in \mathbb{R}, t\in[0,T].$$
\end{itemize}
\end{assum}
\begin{assum}\label{gt:2.2}
		$g(t,x),\sigma(t,x)$ are twice differentiable in $x$ and the derivatives are  bounded by some constant $\M>0$. 
\end{assum}
Let $F, G$ be random variables, we denote by $d_{TV}(F,G)$ the total variation distance between the laws of $F$ and $G$, that is,
\begin{align*}
	d_{TV}(F,G)&=\sup\limits_{A\in \mathcal{B}(\mathbb{R})}|P(F\in A)-P(G\in A)|\\&=\dfrac{1}{2}\sup\{|\phi(F)-\phi(G)|:\phi: \mathbb{R}\to \mathbb{R} \  \mbox{which is bounded by} \ 1\}.
\end{align*}
Consider the following first-order stochastic differential equation, which will be the limiting system for the displacement process
\begin{equation}\label{eqn:1.3}
(\kappa+\gamma)dX_t=-\Big[g(t,X_t)-\dfrac{\gamma}{\kappa}\E[g(t,X_t)]\Big]\,dt+\sigma(t,X_t)dW_t,\quad  X_0=x_0\in \mathbb{R}.
\end{equation}
Our first main result provides an explicit rates of convergence for the displacement process.
\begin{theorem}[Quantitative rates of convergence of the displacement process]
\label{thm: Thm1}
Under Assumptions \ref{gt:2.1} and \ref{gt:2.2}, systems \eqref{eq1} and \eqref{eqn:1.3} have unique solutions and the following statements hold.
\begin{enumerate}[(1)]
\item (rate of convergence in $L^p$-distances) For all $p\geq 2$, $\alpha\geq 1$ and $t\in[0,T]$,
\begin{equation*}
\E \Big[\sup_{0\leq s\leq t}|X^{\alpha}_s-X_s|^{p}\Big] \leq C \left[\left( \lambda(t,\alpha(\kappa+\gamma))\right)^{\frac{p}{2}}+\left( \lambda(t,\alpha\kappa)\right)^p \right],
		\end{equation*}
		where $\lambda(t,a)=(1/a)[1-\exp(-at)]$ for $t, a>0$ and $C$ is a positive constant depending on $\{x_0,y_0,\kappa,\gamma,K,L,p,T\}$  but not on $\alpha$ and $t$.
\item (rate of convergence in the total variation distance). We further assume that $|\sigma(t,x)|\geq \sigma_0>0$ for all $(t,x)\in [0,T]\times \mathbb{R}.$  Then, for each $\alpha\geq 1$ and $t\in (0,T],$ 
		\begin{equation*}
			d_{TV}(X^{\alpha}_t,X_t) \leq C\sqrt{t^{-1}(\lambda(t,\alpha(\kappa+\gamma))+\left( \lambda(t,\alpha\kappa)\right)^2)},
		\end{equation*}
		where $C$ is a constant depending only on $\{x_0,y_0,\sigma_0, \kappa,\gamma,K,L,\,M,T\}$ but not on $\alpha$ and $t$.
	As a corollary, if $|\sigma(t,x)|\geq \sigma_0>0$ for all $(t,x)\in[0,T]\times \mathbb{R}$ then we have.
	\begin{equation*}
	d_{TV}(X^{\alpha}_t,X_t) \leq  Ct^{-1/2}\alpha^{-1/2}.
	\end{equation*}
\end{enumerate}
\end{theorem}
Theorem \ref{thm: Thm1} combines Theorem \ref{bd:2.5} (for the $L^p$-distances) and Theorem \ref{dl:2.1} (for the total variation distance) in Section \ref{sec: displacement process}.

We are also interested in the asymptotic behavior, when $ \alpha \to \infty$, of the velocity process $Y^\alpha_t$ of \eqref{eq1} and of a re-scaled velocity process, $\tilde{Y}^{\alpha}_t$, which is defined by 
$$
\tilde{Y}^{\alpha}_t:=\frac{1}{\sqrt{\alpha}}Y^\alpha_{t/\alpha}.
$$ 
The re-scaled process $\tilde{Y}^{\alpha}_t$ satisfies the following stochastic differential equation 
	\begin{equation}
 \label{eqn:2.10m10}
 \begin{cases}
 \tilde{Y}^{\alpha}_t =\dfrac{y_0}{\sqrt{\alpha}}-(\kappa+\gamma)\int_0^t\tilde{Y}^{\alpha}_sds- \dfrac{1}{\sqrt{\alpha}}\int_0^tg(\frac{s}{\alpha},X^\alpha_{\frac{s}{\alpha}})ds-\gamma \int_0^t\E(\tilde{Y}^{\alpha}_{s})ds+ \int_0^t\sigma(\frac{s}{\alpha},X^\alpha_{\frac{s}{\alpha}})d\tilde{W}_s\\
 X^\alpha_0 = x_0,
 \end{cases}
 \end{equation}
where $\tilde{W}_t:=\sqrt{\alpha}W_{t/\alpha}$ is a rescaled Brownian process.

Now we consider the following stochastic differential equation, which will be the limiting process of the rescaled velocity process
  \begin{equation}
 \label{eqn:2.14m0}
 \begin{cases}
 d\tilde{Y}_t = -(\kappa+\gamma)d\tilde{Y}_t+\sigma(0,x_0)d\tilde{W}_t,\\
 \tilde{Y}(0) = 0.
 \end{cases}
 \end{equation}
We now describe our result for the rescaled velocity process first since for this process we also work with a general setting where both $g$ and $\sigma$ can depend on both spatial and temporal variables. We only assume additionally the following condition.
\begin{assum}\label{gt:2.3}
  		$$ |\sigma(t,x)-\sigma(s,y)| \leq L(|t-s|+|x-y|) \quad \forall x, y\in \mathbb{R}, t,s\in[0,T].$$
  \end{assum}
  
In the next theorem, we provide explicit rates of convergence, both in $L^p$-distances and in the total variation distance, for the rescaled velocity process. 
\begin{theorem}[Quantitative rates of convergence for the rescaled velocity processes]
\label{thm: Thm21}
Under Assumptions \ref{gt:2.1} and \ref{gt:2.3} the following hold.
\begin{enumerate}[(1)]
\item (rate of convergence in $L^p$-distance for the rescaled velocity process) For all $p \geq 2$ and $\alpha\geq 1$,
\begin{equation*}
\E\left[	\sup_{0\leq t\leq T}|\tilde{Y}^\alpha_t-\tilde{Y}_t|^{p}\right] \leq \dfrac{C}{\alpha^{p/2}},
	\end{equation*}
	where $C$ is a positive constant depending on $p$ and other parameters but not on $\alpha$.
\item (rate of convergence in the total variation distance for the rescaled velocity process) Assume that Assumptions \ref{gt:2.1} and \ref{gt:2.3} hold and $|\sigma(0,x_0)|>0$ for all $(t,x)\in [0,T]\times \mathbb{R}.$ Then, for each $\alpha\geq 1$ and $t\in (0,T],$ 
	\begin{equation*}
	d_{TV}(\tilde{Y}^{\alpha}_t,\tilde{Y}_t) \leq C\left(\lambda(t,2(\kappa+\gamma))\right)^{-1/2}\alpha^{-1/2},
	\end{equation*}
	where $C$ is a constant depending only on $\{x_0,y_0,\kappa,\gamma,K,L,p,T,\sigma(0,x_0)\}$.
\end{enumerate}
\end{theorem}
Theorem \ref{thm: Thm21} summarizes Theorem \ref{bd:3.1} (for the $L^p$-distances) and Theorem \ref{thm: TV for velocity} (for the total variation distance) in Section \ref{sec: velocity process}.

When $g(t,x)=g(x)$ and $\sigma(t,x)=\delta$, \cite[Theorem 2.3]{narita3} shows that the velocity process $Y^\alpha_t$ converges to the normal distribution as $\alpha \to \infty$.  The third aim
of this paper is to generalize this result to a more general setting where $g$ depends on both $x$ and $t$ while $\sigma$ depends only on $t$, i.e. $\sigma(t,x)=\sigma(t)$, obtaining rates of convergence in the total variation distance. The following theorem is the content of Theorem \ref{thm: TV for velocity2} in Section \ref{sec: velocity process}.

\begin{theorem}[Quantitative rates of convergence for the velocity processes]
	\label{thm: Thm2}
	Under Assumptions \ref{gt:2.1}  the following hold.  Assume additionally that $\sigma(t)$ is continuously differentiable on $[0,T]$ and that $\sigma(t)\not=0$ for each $t \in (0, T]$. Let $N$ be a normal random variable with mean $0$ and variance $\dfrac{\sigma^2(t)}{2(\kappa+\gamma)}$,
 Then, for each $\alpha\geq 1$ and $t\in (0,T]$ 
		\begin{equation*}
		d_{TV}\left(\dfrac{Y^{\alpha}_t}{\sqrt{\alpha}},N\right) \leq C\left(\lambda(t,2(\kappa+\gamma))\right)^{-1/2}\alpha^{-1/2},
		\end{equation*}
		where $C>0$ is a constant not depending  on $\alpha$ and $t$. 
\end{theorem}
Theorem \ref{thm: Thm2} is Theorem \ref{thm: TV for velocity2}  in Section \ref{sec: velocity process}.
We emphasize that in the main theorems, to obtain the existence and uniqueness as well as the rate of convergence in $L^p$-distances we only use Assumptions \ref{gt:2.1}. Assumptions \ref{gt:2.2} and \ref{gt:2.3} are needed to employ techniques from Malliavin calculus, in particular to derive estimates for the Malliavin derivatives.
\begin{corollary}[Rate of convergence in Wasserstein distance for the laws of the displacement and velocity processes] Let $\mu$ and $\nu$ be two probability measures with finite second moments, then the $p$-Wasserstein distance, $W_p(\mu,\nu)$, between them can be defined by
$$
W_p(\mu,\nu)=\Big(\inf\{\E\big[|X-Y|^p\big]: X\sim \mu, Y\sim \nu\}\Big)^{1/p}.
$$
Using this formulation, as a direct consequence of our main results, we also obtain explicit rates of convergence in $p$-Wasserstein distances for the laws of the displacement and the rescaled velocity processes to the corresponding limiting ones
\begin{align*}
&\sup_{t\in[0,T]}W_p^p(\mathrm{law}(X^\alpha_t),\mathrm{law}(X_t))  \leq \frac{C}{\alpha^{p/2}},
\\&\sup_{t\in[0,T]}W_p^p(\mathrm{law}(\tilde{Y}^\alpha_t),\mathrm{law}(\tilde{Y}_t)) \leq \frac{C}{\alpha^{p/2}}.
\end{align*}
\end{corollary}
\subsection{Comparison with existing literature and future work} 
The zero-mass limit of second order differential equations has been studied intensively in the literature. In the seminal work~\cite{Kramers1940}, Kramers formally discusses this problem, in the context of applications to chemical reactions, for the classical underdamped Langevin dynamics, which corresponds to~\eqref{eq1} with $g=-\nabla V$ (a gradient potential force), $\gamma= 0$ (no interaction force) and a constant diffusion coefficient. Due to this seminal work, this limit has become known in the literature as the Smoluchowski-Kramers approximation. Nelson rigorously shows that, under suitable rescaling, the solution to the Langevin equation converges almost surely to the solution of~\eqref{eq1m} with $\psi=0$ ~\cite{Nelson1967}. Since then  various generalizations and related results have been proved using different approaches  such as stochastic methods, asymptotic expansions and variational techniques, see for instance \cite{narita1,narita2,narita3,Freidlin2004a,Cerrai2006,Hottovy2012,Duong2017,Duong2018,Nguyen2020}. The most relevant papers to the present one include~\cite{narita1,narita2,narita3,Duong2017,Nguyen2020}. The main novelty of the present paper lies in the fact that we consider interacting (mean-field) systems allowing time-dependent external forces and diffusion coefficients, and providing explicit rates of convergence in both $L^p$-distances and total variation distances for both displacement and velocity processes. Existing papers lack at least one of these features. More specifically,

\textit{Papers that consider mean-field (interaction) systems}. The papers \cite{narita1,narita2,narita3,Duong2017} consider second order mean-field stochastic differential equations establishing the zero-mass limit, but they require much more stringent conditions that $g(t, x) = g(x)$ (time-independent force) and $\sigma(t, x) = \delta$ (constant diffusivity). On top of that, they do not provide a rate of convergence. Furthermore, our approach using Malliavin calculus is also different: Narita's papers use direct arguments while \cite{Duong2017} employs variational methods based on Gamma-convergence and large deviation principle.

\textit{Papers that provide a rate of convergence}. The papers \cite{Nguyen2020, Duong2018} provide a rate of convergence but only consider non-interacting systems (also using different measurements). Like our paper, \cite{Nguyen2020} also utilizes techniques from Malliavin calculus, but \cite{Duong2018} uses a completely different variational method. The recent paper \cite{Choi2022}, which studies the kinetic Vlasov-Fokker-Planck equation, is particularly interesting since it considers both interacting systems and provides a rate of convergence, but this paper is different to ours in a couple of aspects. First, the interaction force is acting on the position instead of the velocity; second, it works on the Fokker Planck equations and obtains a rate of convergence in Wasserstein distance while we work on the stochastic differential equations and obtain error quantifications in both $L^p$-distances  and total variation distances; third, as mentioned, we use Malliavin calculus while \cite{Choi2022} applied variational techniques like in \cite{Duong2017, Duong2018}. We also mention the paper \cite{Son2020}, which provides similar rate of convergence to ours but it consider non mean-field stochastic differential equations driven by fractional Brownian motions. 

\textit{Future work}. The Lipschitz boundedness and differentiability Assumptions \ref{gt:2.1}-\ref{gt:2.2}-\ref{gt:2.3} are standard, but rather restricted since they do not cover some physically interesting interacting singular, such as Coloumb or Newton, forces. It would be interesting and challenging to extend our work to non-Lipschtizian and singular coefficients. Initial attempts in this direction for related models exist in the literature, see \cite{breimhorst2009smoluchowski} for non-Lipschitzian coefficients and recent papers \cite{XIE2022,Choi2022} for singular forces. Another interesting problem for future work is to study the Kramers-Smoluchowski approximation for the $N$-particle system \eqref{eq1n} obtaining a rate of convergence that is independent of $N$.
\subsection{Overview of the proofs}
To prove the main theorems for the general setting, with time-dependent coefficients, and obtain $L^p$-distances and total variations distances for the position and velocity processes, several technical improvements have been carried out.

\textit{On existence and uniqueness.} Under Assumptions \ref{gt:2.1}, the existence and uniqueness, as well as the boundedness of the moments, of the second-order system \eqref{eq1} and the limiting first-order one \eqref{eqn:1.3} are standard results following H\"{o}lder's and the Burkholder-Davis-Gundy inequalities.

\textit{On rate of convergence in $L^p$-distances.} Combining the mentioned inequalities and known estimates from \cite{narita1} we can directly estimate $\E \Big[\sup_{0\leq s\leq t}|X^{\alpha}_s-X_s|^{p}\Big] $ and $\E\left[\sup_{0\leq t\leq T}|\tilde{Y}^\alpha_t-\tilde{Y}_t|^{p}\right]$ and obtain  the rate of convergences in $L^p$-distances, proving parts (1) of both theorems.

\textit{On rate of convergence in total-variation distances}. The Malliavin differentiablity of the processes is followed from similar arguments as in \cite{nualartm2}. Obtaining the rate of convergence in total variation distances is the most technically challenging. Lemma \ref{bd:2.1}, which provides an upper bound estimate for the total variation between two random variables in terms of their Malliavin derivatives, is the key in our analysis. This  lemma enables us to obtain the desired rates of convergence by estimating the corresponding quantities appearing in the  right-hand side of Lemma \ref{bd:2.1}.

\subsection{Organization of the paper}	
	The rest of of the paper is organized as follows. In Section \ref{sec: review}, we give an overview of some elements of Malliavin calculus and mean-field stochastic differential equations. The proofs of the main theorems are given in Section \ref{sec: proof}. 
	
\section{Preliminaries}
\label{sec: review}
In this section, we provide some basic and directly relevant knowledge on the Malliavin calculus and mean-field stochastic differential equations.
\subsection{Malliavin calculus}
	Let us recall some elements of stochastic calculus of variations (for more details see \cite{nualartm2}). We suppose that $(W_t)_{t\in [0,T]}$ is defined on
	a complete probability space $(\Omega,\mathcal{F},\mathbb{F},P)$, where $\mathbb{F}=(\mathcal{F}_t)_{t\in [0,T]}$ is a natural filtration generated by
	the Brownian motion $W.$ For $h\in L^2[0,T]:=\mathcal{H},$ we denote by $W(h)$ the Wiener integral
	$$W(h)=\int\limits_0^T h(t)dW_t.$$
	Let $\mathcal{S}$ denote the dense subset of $L^2(\Omega, \mathcal{F},P):=L^2(\Omega)$ consisting of smooth random variables of the form
	\begin{equation}\label{ro}
		F=f(W(h_1),...,W(h_n)),
	\end{equation}
	where $n\in \mathbb{N}, f\in C_b^\infty(\mathbb{R}^n),h_1,...,h_n\in \mathcal{H}.$ If $F$ has the form (\ref{ro}), we define its Malliavin derivative as the process $DF:=\{D_tF, t\in [0,T]\}$ given by
	$$D_tF=\sum\limits_{k=1}^n \frac{\partial f}{\partial x_k}(W(h_1),...,W(h_n)) h_k(t).$$
	More generally, for each $k\geq 1$ we can define the iterated derivative operator on a cylindrical random variable by setting
	$$D^{k}_{t_1,...,t_k}F=D_{t_1}...D_{t_k}F.$$
	For any $p,k\geq 1,$ we shall denote by $\mathbb{D}^{k,p}$ 
	the closure of $\mathcal{S}$ with respect to the norm
	$$\|F\|^p_{k,p}:=\E\big[|F|^p\big]+\E\bigg[\int_0^T|D^{1}_{t_1}F|^pdt_1\bigg]+...+\E\bigg[\int_0^T...\int_0^T|D^{k}_{t_1,...,t_k}F|^pdt_1...dt_k\bigg].$$
	A random variable $F$ is said to be Malliavin differentiable if it belongs to $\mathbb{D}^{1,2}.$
	
	An important operator in the Malliavin's calculus theory is the divergence operator $\delta$, which is the adjoint of the derivative operator $D$. The domain of $\delta$ is the set of all functions $u\in L^2(\Omega, \mathcal{H})$ such that 
	$$\E\big[|\langle DF, u \rangle_{\mathcal{H}}|\big]\leq C(u)\|F\|_{L^2(\Omega)},$$  
	where $C(u)$ is some positive constant depending on $u$. In particular, if $u\in \mathrm{Dom}(\delta)$, then $\delta(u)$ is characterized by the following duality relationship
	$$\E\big[\langle DF, u \rangle_{\mathcal{H}}\big]=\E[F\delta(u)].$$
 
The following lemma provides an upper bound on the total variation distance between two random variables in terms of their Malliavin derivatives. This lemma will play an important role in the analysis of the present paper.
	\begin{lemma}\label{bd:2.1}
		Let $F_1\in \mathbb{D}^{2,2}$ be such that 	 $\|DF_1\|_{\mathcal{H}}>0\,\,a.s.$ Then, for any random variable $F_2\in \mathbb{D}^{1,2}$ we have
	\begin{align}\label{pt2} d_{TV}(F_1,F_2)
		\leq \|F_1-F_2\|_{1,2} \left[3\left(\E\|D^2F_1\|^4_{\mathcal{H}\bigotimes\mathcal{H}} \right)^{1/4}\left( \E\|DF_1\|_\mathcal{H}^{-8}\right)^{1/4}+ 2\left(\E\|DF_1\|_\mathcal{H}^{-2}\right)^{1/2}\right],
		\end{align}
		provided that the expectations exist.
	\end{lemma}
\begin{proof}
	From \cite[Lemma 2.1]{Son2020} we have
	\begin{align}\label{pt1}
 d_{TV}(F_1,F_2)
\leq\|F_1-F_2\|_{1,2} \left[\left(\E\delta\left(\dfrac{DF_1 }{\|DF_1\|_\mathcal{H}^2}\right)^2\right)^{1/2}+ \left(\E\|DF_1\|_\mathcal{H}^{-2}\right)^{1/2}\right].
	\end{align}
	Now using \cite[Proposition 1.3.1]{nualartm2}, we get
\begin{align}\label{2eq}
\E\delta\left(\dfrac{DF_1 }{\|DF_1\|_\mathcal{H}^2}\right)^2&\leq \E\left\|\dfrac{DF_1}{\|DF_1\|_{\mathcal{H}}^2}\right\|^2_{\mathcal{H}}+\E\left\|D\left(\dfrac{DF_1 }{\|DF_1\|_\mathcal{H}^2}\right)\right\|_{ \mathcal{H}\bigotimes\mathcal{H}}^2\notag\\&=\E\left\|DF_1\right\|^{-2}_{\mathcal{H}}+\E\left\|D\left(\dfrac{DF_1 }{\|DF_1\|_\mathcal{H}^2}\right)\right\|_{ \mathcal{H}\bigotimes\mathcal{H}}^2.
\end{align}	
Moreover, observing that
$$D\left(\dfrac{DF_1 }{\|DF_1\|_\mathcal{H}^2}\right)=\dfrac{D^2F_1 }{\|DF_1\|_\mathcal{H}^2}-2\dfrac{\langle D^2F_1,DF_1\bigotimes DF_1\rangle_{\mathcal{H}\bigotimes\mathcal{H}} }{\|DF_1\|_\mathcal{H}^4},$$
which implies that
\begin{equation}\label{1eq}
\left\|D\left(\dfrac{DF_1 }{\|DF_1\|_\mathcal{H}^2}\right)\right\|_{ \mathcal{H}\bigotimes\mathcal{H}}\leq \dfrac{3\|D^2F_1\|_{\mathcal{H}\bigotimes\mathcal{H}} }{\|DF_1\|_\mathcal{H}^2}.
\end{equation}
Substituting the inequality \eqref{1eq} into \eqref{2eq} and using H\"older's inequality, one can derive that
\begin{align*}
\E\delta\left(\dfrac{DF_1 }{\|DF_1\|_\mathcal{H}^2}\right)^2&\leq \E\left\|DF_1\right\|^{-2}_{\mathcal{H}}+9\E\left(\dfrac{\|D^2F_1\|^2_{\mathcal{H}\bigotimes\mathcal{H}} }{\|DF_1\|_\mathcal{H}^4}\right)\\&\leq \E\left\|DF_1\right\|^{-2}_{\mathcal{H}}+9\left(\E\|D^2F_1\|^4_{\mathcal{H}\bigotimes\mathcal{H}} \right)^{1/2}\left( \E\|DF_1\|_\mathcal{H}^{-8}\right)^{1/2}.
\end{align*}	
Finally, substituting the above estimate back into \eqref{pt1} and using  the fundamental inequality $(a+b)^{1/2}\leq a^{1/2}+b^{1/2}$ for all $a,b\geq 0$, we obtain \eqref{pt2}, which completes the proof of this lemma.
\end{proof}
\subsection{Mean-field stochastic differential equations}
Let $(\Omega, \mathcal{F}, \mathbb{P})$ be a probability space with an increasing family $\{\mathcal{F}_t; t\geq 0\}$ of sub-$\sigma$-algebras of $\mathcal{F}$ and let $\{W_t; t\geq 0\}$ be a one-dimensional Brownian motion process adapted to $\mathcal{F}_t$.
	
The following lemma provides equivalent formulations of \eqref{eq1} and \eqref{eqn:1.3} as stochastic integral equations. 
	
	\begin{lemma} 
		\label{bd:2.2}
		Equations \eqref{eq1} and \eqref{eqn:1.3} are, respectively, equivalent to the following equations
\begin{align}
X^{\alpha}_t &= x_0 - \dfrac{1}{\kappa+\gamma}\int_{0}^{t}g(s,X^{\alpha}_s)\ ds - \dfrac{\gamma}{\kappa+\gamma}\int_{0}^{t}G^{\alpha}(s)\ ds + \dfrac{1}{\kappa+\gamma}\int_{0}^t\sigma(s,X^\alpha_s)dW_s\notag\\
				& \qquad + I_{0}^{\alpha}(t) + I_{1}^{\alpha}(t) - I_{2}^{\alpha}(t) - I_{3}^{\alpha}(t) - I_{4}^{\alpha}(t),\label{eqn:2.2}\\
\notag \\		
X_t &= x_0 - \dfrac{1}{\kappa+\gamma}\int_{0}^{t}g(s,X_s)\ ds - \dfrac{\gamma}{\kappa+\gamma}\int_{0}^{t}G(s)\ ds + \dfrac{1}{\kappa+\gamma}\int_{0}^t\sigma(s,X_s)dW_s,	\label{eqn:2.3}
\end{align}
where $G^{\alpha}(t) = \E[g(t,X^{\alpha}_t)]/\kappa$, $G(t) = \E[g(t,X_t)]/\kappa$, and 
		\begin{align*}
			&I_0^{\alpha}(t) = y_0\lambda(t;\alpha(\kappa+\gamma)),\\
			&I_1^{\alpha}(t) = \dfrac{1}{\kappa+\gamma}\int_{0}^{t} \exp{[\alpha(\kappa+\gamma)(u-t)]}g(u,X^{\alpha}_u)du,\\
			&I_2^{\alpha}(t) = \dfrac{\gamma}{\kappa+\gamma}\int_{0}^{t} \exp{[\alpha(\kappa+\gamma)(u-t)]}n^{\alpha}(u)\ du,\\
			&I_3^{\alpha}(t) = \dfrac{1}{\kappa+\gamma}\int_{0}^{t} \exp{[\alpha(\kappa+\gamma)(u-t)]}\sigma(u,X^\alpha_u) dW_u,\\
			&I_4^{\alpha}(t) = \dfrac{\gamma}{\kappa+\gamma}\left(y_0\lambda(t;\alpha\kappa)+ \int_{0}^{t} \exp{[\alpha\kappa(u-t)]}G ^{\alpha}(u)du\right),
		\end{align*}
		where $\lambda(t;a) = (1/a)(1-\exp{[-at]}), n^{\alpha}(t) = E[Y^\alpha_t]$. 
	\end{lemma}
	\begin{proof} Firstly, we can rewrite the second equation of \eqref{eq1} as follows
		\begin{align*}{}&dY^\alpha_t=[-\alpha\kappa Y^\alpha_t-\alpha g(t,X^\alpha_t)-\alpha\gamma(Y^\alpha_t-n^\alpha(t))]dt+\alpha \sigma(t,X^\alpha_t)dW_t \ \mbox{with } \ Y^\alpha_0=y_0.
		\end{align*}
		Using It\^o formula,  we have the following expression
		\begin{align*}d(\exp{[\alpha(\kappa+\gamma)t]}Y^\alpha_t)&=-\alpha\exp{[\alpha(\kappa+\gamma)t]}g(t,X_t^\alpha)dt+\alpha\gamma\exp{[\alpha(\kappa+\gamma)t]}n^\alpha(t)dt
		\\&\qquad+\alpha\exp{[\alpha(\kappa+\gamma)t]}\sigma(t,X_t^\alpha)dW_t,
		\end{align*}
which implies
		\begin{align}\label{eqmm1}Y^\alpha_t&=y_0\exp{[-\alpha(\kappa+\gamma)t]}-\alpha\int_{0}^{t} \exp{[\alpha(\kappa+\gamma)(s-t)]}g(s,X^{\alpha}_s)ds\notag\\&\qquad+\alpha\gamma\int_{0}^{t} \exp{[\alpha(\kappa+\gamma)(s-t)]}n^{\alpha}(s)\ ds
			+\alpha\int_{0}^{t} \exp{[\alpha(\kappa+\gamma)(s-t)]}\sigma(s,X^\alpha_s) dW_s.
		\end{align}
		Secondly, substituting this equation into the first equation of \eqref{eq1} we get
		\begin{align*}X^\alpha_t&=x_0+y_0\int_0^t\exp{[-\alpha(\kappa+\gamma)s]}ds-\alpha\int_{0}^{t}\int_0^s \exp{[\alpha(\kappa+\gamma)(u-s)]}g(u,X^{\alpha}_u)duds\\&+\alpha\gamma\int_{0}^{t}\int_0^s \exp{[\alpha(\kappa+\gamma)(u-s)]}n^{\alpha}(u)\ duds
			+\alpha\int_{0}^{t}\int_0^s \exp{[\alpha(\kappa+\gamma)(u-s)]}\sigma(u,X^\alpha_u) dW_uds.
		\end{align*}
		Now, we use integration by parts for the non-stochastic integral and Ito's product rule for the stochastic integral to get 
		\begin{align}
		X^{\alpha}_t &= x_0 - \dfrac{1}{\kappa+\gamma}\int_{0}^{t}g(u,X^{\alpha}_u)\ du + \dfrac{\gamma}{\kappa+\gamma}\int_{0}^{t}n^{\alpha}(u)\ du + \dfrac{1}{\kappa+\gamma}\int_{0}^t\sigma(s,X^\alpha)dW(s)\notag\\
			& \ + I_{0}^{\alpha}(t) + I_{1}^{\alpha}(t) - I_{2}^{\alpha}(t) - I_{3}^{\alpha}(t),\label{eq:Xalpha}
		\end{align}
where the terms $I_i^{\alpha}(t)$ ($i=0,1,2,3$) are defined in the statement of the lemma.
		On the other hand, from  the second equation of \eqref{eq1} we have $$
		\dfrac{d}{dt}n^\alpha(t)=-\alpha\kappa n^\alpha(t)-\alpha\kappa G^\alpha(t) \ \ \mbox{with} \ n^\alpha(0)=y_0.$$
		This implies that	\begin{align}\label{eqm1m}n^\alpha(t)=y_0\exp{[-\alpha\kappa t]}-\alpha\kappa\int_0^t\exp{[\alpha\kappa(u-t)]}G^\alpha(u)du.\end{align}
		Integrating this equation over the interval $[0,t]$ and changing the order of integration in the double integral, we get
		\begin{align}\label{eqm1}\int_0^t n^\alpha(s)\,ds = y_0\lambda(t,\alpha\kappa) -\int_0^tG^\alpha(s)ds+\int_0^t\exp[\alpha\kappa(s-t)]G^\alpha(s)ds.
		\end{align}
Substituting \eqref{eqm1} back into \eqref{eq:Xalpha} we obtain \eqref{eqn:2.2}.
		
		The proof for Equation (\ref{eqn:2.3}) is similar.
	\end{proof}

The  existence and uniqueness of solutions to \eqref{eqn:2.2} and \eqref{eqn:2.3} under Assumptions \ref{gt:2.1} is stated in the \cite{mckean1967propagation}. 
	
	\section{Proof of the main results}
	\label{sec: proof}
In this section, we present the proofs of the main theorems \ref{thm: Thm1}, \ref{thm: Thm21} and \ref{thm: Thm2}. We start with the displacement process (Theorems \ref{bd:2.5} and \ref{dl:2.1} give Theorem \ref{thm: Thm1}) in Section \ref{sec: displacement process}. Then in Section \ref{sec: velocity process} we deal with the rescaled velocity process and the velocity process (Theorems \ref{bd:3.1} and \ref{thm: TV for velocity} give Theorem \ref{thm: Thm21} and Theorem \ref{thm: Thm2} is Theorem \ref{thm: TV for velocity2}).
	\subsection{Approximation of the displacement process}
	\label{sec: displacement process}
	In this section, we give explicit bounds on $L^p$-distances and the total variation distance between the solution $X^\alpha_t$  of \eqref{eq1} and the solution $X_t$ of \eqref{eqn:1.3}.  We will repeatedly use the following fundamental inequalities.
\begin{enumerate}[(i)]
\item Minkowski's inequality: for $p\geq 1$ and $n$ real numbers $a_1,\ldots, a_n$, we have
\begin{equation}
\label{eq: Minkowski inequality}
\Big|\sum_{i=1}^n a_i\Big|^p \leq n^{p-1}\sum_{i=1}^n |a_i|^p.
\end{equation}
\item H\"older's inequality: for $p\geq 1$, $t>0$ and  measurable functions $f$ we have
\begin{equation}
\label{eq: Holder inequality}
\Big(\int_0^t |f(s)|\,ds\Big)^p\leq t^{p-1}\int_{0}^t|f(s)|^p\,ds. 
\end{equation}
\item The Burkholder-Davis-Gundy (BDG) inequality for Brownian stochastic integrals, see for instance \cite[Section 17.7]{Schilling}: for $0<p<\infty$ and $f\in L^2([0,t],\Omega)$ we have
\begin{equation}
\label{eq: BDG}
\E\left[\sup_{s\in[0,t]}\Big|\int_0^s  f_r\, dW_r\Big|^p\right]\leq C_p\E\left[\Big(\int_0^t|f_s|^2\,ds\Big)^{p/2}\right],
\end{equation}
where $C_p$ is a positive constant depending only on $p$.
\end{enumerate}
Applying the BDG inequality \eqref{eq: BDG} to solutions of \eqref{eqn:2.2} and \eqref{eqn:2.3} we obtain
\begin{align}
\label{eq: BDG inequality}
&\E\left[\sup_{s\in[0,t]}\Big|\int_0^s \sigma(r,X^\alpha_r)\, dW_r\Big|^p\right]\leq C_p\E\left[\Big(\int_0^t|\sigma(s,X^\alpha_s)|^2\,ds\Big)^{p/2}\right],
\\&\E\left[\sup_{s\in[0,t]}\Big|\int_0^s (\sigma(r,X^\alpha_r)-\sigma(r,X_r))\, dW_r\Big|^p\right]\leq C_p\E\left[\Big(\int_0^t|\sigma(s,X^\alpha_s)-\sigma(s,X_s)|^2\,ds\Big)^{p/2}\right].\label{eq: BDG inequality2} 
\end{align}
The next lemma provides important estimates on the moments of the displacement process $\{X^\alpha_t, t\in [0,T]\}$, which will be helpful to prove the main results of this section. 
	Hereafter, we denote by $C$  a generic constant which may vary at each
	appearance.	
	\begin{lemma}
		\label{bd:2.4}
		Let $\{X^\alpha(t), t\in [0,T]\}$ be the solution of \eqref{eqn:2.2} under Assumptions \ref{gt:2.1}. Then, for all $p \geq 2$,
		\begin{equation}\label{eq1mm}
			\sup_{\alpha>0}\E\Big[\sup_{0\leq t\leq T}|X^{\alpha}(t)|^{p}\Big] \leq C,
		\end{equation}
		and for all \ $ 0\leq t\leq T$,
			\begin{equation}\label{eq2mm}
		\sup_{\alpha>0}\E\Big[|X^{\alpha}(t)-x_0|^p\Big] \leq Ct^{p/2},
		\end{equation}
		where $C$ is a positive constant depending only on $\{x_0,y_0,\kappa,\gamma,K,p,T\}$.	
	\end{lemma}
\begin{proof}
We first prove \eqref{eq1mm}. We shall divide the proof into two steps.

{\bf Step 1:} We evaluate the upper bound of the moments of each $I_i^\alpha(t), i=1,2,3,4$.

From definition of $I_1^\alpha(t)$ and Assumptions \ref{gt:2.1} we have 
\begin{align*}\sup\limits_{0\leq s\leq t}|I_1^\alpha(s)|&\leq  \dfrac{K}{\kappa+\gamma}\sup\limits_{0\leq s\leq t}\int_{0}^{s} \exp{[\alpha(\kappa+\gamma)(u-s)]}(1+|X^{\alpha}_u|)du\\&\leq 
\dfrac{Kt}{\kappa+\gamma}+\dfrac{K}{\kappa+\gamma}\int_0^t\sup\limits_{0\leq u\leq s}|X^{\alpha}_u|ds.
\end{align*}
Now Minkowski's inequality \eqref{eq: Minkowski inequality} with $n=2$ and  H\"older's inequality \eqref{eq: Holder inequality} yield
\begin{align}\label{eq1.m}\E\Big(\sup\limits_{0\leq s\leq t}|I_1^\alpha(s)|^p\Big)\leq 
\dfrac{2^{p-1}K^pt^p}{(\kappa+\gamma)^p}+\dfrac{2^{p-1}K^pt^{p-1}}{(\kappa+\gamma)^p}\int_0^t\E\Big(\sup\limits_{0\leq u\leq s}|X^\alpha_u|^p\Big)ds
\end{align}
For $I_2^\alpha(t),$ by substituting \eqref{eqm1m} into $I_2^\alpha(t)$ and changing the order of integration in the double integral, one can derive that
\begin{align}\label{ptme}I_2^\alpha(t)&=\dfrac{y_0\gamma\lambda(t,\alpha\gamma)}{\kappa+\gamma}\exp[-\alpha(\kappa+\gamma)t]-\dfrac{1}{\kappa+\gamma}\int_{0}^{t} \exp{[\alpha\gamma(s-t)]}E[g(s,X_s^\alpha)]ds\notag\\&\qquad \qquad\qquad\qquad+ 
\dfrac{1}{\kappa+\gamma}\int_{0}^{t} \exp{[\alpha(\kappa+\gamma)(s-t)]}E[g(s,X_s^\alpha)]ds.
\end{align}
Using Assumptions \ref{gt:2.1}, Minkowski's inequality \eqref{eq: Minkowski inequality} with $n=2$ and H\"older's inequality \eqref{eq: Holder inequality}  with noting that $\lambda(t,a)\leq t$ for all $t\geq 0, a\geq 0$, we get
\begin{align}\label{eq1.mm}
\sup\limits_{0\leq s\leq t}|I_2^\alpha(s)|^p&\leq \left(\dfrac{y_0\gamma t}{\kappa+\gamma}+\dfrac{2K}{\kappa+\gamma}\int_{0}^{t} (1+\E|X_s^\alpha|)ds\right)^p\notag\\&\leq \dfrac{2^{p-1}\left(y_0\gamma+ 2K\right)^pt^p}{(\kappa+\gamma)^p}+\dfrac{2^{2p-1}K^pt^{p-1}}{(\kappa+\gamma)^p}\int_0^t\E(\sup\limits_{0\leq u\leq s}|X^\alpha_u|^p)ds.
\end{align}
For $I_3^\alpha$, using the BDG inequality \eqref{eq: BDG inequality}, H\"older's inequality \eqref{eq: Holder inequality} and  Assumptions \ref{gt:2.1}, we get
\begin{align}\label{eq1.m1}
\E\Big(\sup\limits_{0\leq s\leq t}|I_3^\alpha(s)|^p\Big)&\leq \dfrac{C_p}{(\kappa+\gamma)^p}\E\left(\int_{0}^t\exp{[2\alpha(\kappa+\gamma)(s-t)]}|\sigma(s,X^\alpha_s)|^2ds\right)^{p/2}\notag\\&\leq\dfrac{2^{p-1}K^pt^{\frac{p-2}{2}}C_p}{(\kappa+\gamma)^p}\int_{0}^t\exp{[p\alpha(\kappa+\gamma)(s-t)]}(1+\E|X^{\alpha}_s|^p) ds\notag\\&\leq\dfrac{2^{p-1}K^pt^{\frac{p}{2}}C_p}{(\kappa+\gamma)^p}+\dfrac{2^{p-1}K^pt^{\frac{p-2}{2}}C_p}{(\kappa+\gamma)^p}\int_{0}^t\E\Big(\sup\limits_{0\leq u\leq s}|X^\alpha_u|^p\Big) ds.
\end{align}
Next,  from Assumptions \ref{gt:2.1}, Minkowski's inequality \eqref{eq: Minkowski inequality} with $n=2$ and H\"older's inequality \eqref{eq: Holder inequality}  with noting that $\lambda(t,a)\leq t$ for all $t\geq 0, a\geq 0$, we get
\begin{align}\label{eq1.mmm}
\sup\limits_{0\leq s\leq t}|I_4^\alpha(s)|^p&\leq \dfrac{\gamma^p}{(\kappa+\gamma)^p}\left[y_0t+ \dfrac{K}{\kappa}\int_{0}^{t}\left(1+\E\Big(\sup\limits_{0\leq u\leq s}|X^\alpha_u|\Big)ds\right) ds\right]^p
\notag\\&\leq  \dfrac{2^{p-1}\gamma^p}{(\kappa+\gamma)^p}\left[\left(y_0+\dfrac{K}{\kappa}\right)^pt^p+\dfrac{K^pt^{p-1}}{\kappa^p}\int_0^t\E\Big(\sup\limits_{0\leq u\leq s}|X^\alpha_u|^p\Big)ds\right].
	\end{align}
{\bf Step 2:} We estimate the integrand in the integrals in the right hand side of the above expressions. From \eqref{eqn:2.2}, applying Minkowski's inequality with $n=9$, H\"older's inequality \eqref{eq: Holder inequality},  the BDG inequality \eqref{eq: BDG inequality} as well as Assumptions \ref{gt:2.1}, we obtain
	\begin{align*}
	\E\Big(\sup\limits_{0\leq s\leq t}|X^{\alpha}_s|^{p}\Big) &\leq 9^{p-1}\Big[x_0^{p}+ \dfrac{t^{p-1}}{(\kappa+\gamma)^p}\int_{0}^{t}\E|g(s,X^{\alpha}_s)|^p ds + \dfrac{\gamma^pt^{p-1}}{(\kappa+\gamma)^p}\int_{0}^{t}|G^{\alpha}(s)|^p\ ds \\
	&\qquad+ \dfrac{C_p}{(\kappa+\gamma)^p}\E\left(\int_{0}^t|\sigma(s,X^\alpha_s)|^2ds\right)^{p/2} + \sup\limits_{0\leq s\leq t}|I_{0}^{\alpha}(s)|^p + \E\Big(\sup\limits_{0\leq s\leq t}|I_{1}^{\alpha}(s)|^p\Big)
\\&\qquad + \sup\limits_{0\leq s\leq t}|I_{2}^{\alpha}(s)|^p+ \E\Big(\sup\limits_{0\leq s\leq t}|I_{3}^{\alpha}(s)|^p\Big)+ \sup\limits_{0\leq s\leq t}|I_{4}^{\alpha}(s)|^p\Big]
	\\&\leq 9^{p-1}\Big[x_0^{p}+ 2^{p-1}K^p\left(\dfrac{t^{p-1}}{(\kappa+\gamma)^p} + \dfrac{\gamma^pt^{p-1}}{\kappa^p(\kappa+\gamma)^p}\right)\int_{0}^{t}(1+\E|X^{\alpha}_s|^p) ds\\
	&\qquad+ \dfrac{2^{p-1}K^pt^{\frac{p-2}{2}}C_p}{(\kappa+\gamma)^p}\int_{0}^t(1+\E|X_s^\alpha|^p)ds + \sup\limits_{0\leq s\leq t}|I_{0}^{\alpha}(s)|^p + \E\Big(\sup\limits_{0\leq s\leq t}|I_{1}^{\alpha}(s)|^p\Big) 
\\&\qquad+ \sup\limits_{0\leq s\leq t}|I_{2}^{\alpha}(s)|^p+ \E\Big(\sup\limits_{0\leq s\leq t}|I_{3}^{\alpha}(s)|^p\Big)+ \sup\limits_{0\leq s\leq t}|I_{4}^{\alpha}(s)|^p\Big].
\end{align*}
From this, together with \eqref{eq1.m}, \eqref{eq1.mm}, \eqref{eq1.m1} and \eqref{eq1.mmm}, we deduce that
	\begin{equation}
	\label{eqn:2.8}
		\E\Big(\sup\limits_{0\leq s\leq t}|X^{\alpha}_s|^{p}\Big) \leq C+C\int_{0}^{t}	\E\Big(\sup\limits_{0\leq u\leq s}|X^{\alpha}_u|^{p}\Big)ds,
	\end{equation}
	where $C$ is a positive constant depending on $\{x_0,y_0,\kappa,\gamma,K,p,T\}$. From \eqref{eqn:2.8}, by applying Gronwall's lemma, we have
	\begin{equation*}
	\sup_{\alpha>0}\E\Big[\sup_{0\leq t\leq T}|X^{\alpha}(t)|^{p}\Big] \leq C,
	\end{equation*}
which completes the proof of \eqref{eq1mm}. 
	
Next we prove \eqref{eq2mm}. From expression  \eqref{eqn:2.2}, applying Minkowski's inequality \eqref{eq: Minkowski inequality} with $n=8$, H\"older's inequality \eqref{eq: Holder inequality}, the BDG inequality \eqref{eq: BDG inequality} and Assumptions \ref{gt:2.1} again, we get
		\begin{align*}
	\E|X^{\alpha}_s-x_0|^{p} &\leq 8^{p-1}\Big[ \dfrac{t^{p-1}}{(\kappa+\gamma)^p}\int_{0}^{t}\E|g(s,X^{\alpha}_s)|^p ds + \dfrac{\gamma^pt^{p-1}}{(\kappa+\gamma)^p}\int_{0}^{t}|G^{\alpha}(s)|^p\ ds \\
	&\qquad+ \dfrac{C_p}{(\kappa+\gamma)^p}\E\left(\int_{0}^t|\sigma(s,X^\alpha_s)|^2ds\right)^{p/2} + \sup\limits_{0\leq s\leq t}|I_{0}^{\alpha}(s)|^p + \E\Big(\sup\limits_{0\leq s\leq t}|I_{1}^{\alpha}(s)|^p\Big)
\\&\qquad+ \sup\limits_{0\leq s\leq t}|I_{2}^{\alpha}(s)|^p+ \E\Big(\sup\limits_{0\leq s\leq t}|I_{3}^{\alpha}(s)|^p\Big)+ \sup\limits_{0\leq s\leq t}|I_{4}^{\alpha}(s)|^p\Big]
	\\&\leq 8^{p-1}\Big[ 2^{p-1}K^p\left(\dfrac{t^{p-1}}{(\kappa+\gamma)^p} + \dfrac{\gamma^pt^{p-1}}{\kappa^p(\kappa+\gamma)^p}\right)\int_{0}^{t}(1+\E|X^{\alpha}_s|^p) ds\\
	&\qquad + \dfrac{2^{p-1}K^pt^{\frac{p-2}{2}}C_p}{(\kappa+\gamma)^p}\int_{0}^t(1+\E|X_s^\alpha|^p)ds + \sup\limits_{0\leq s\leq t}|I_{0}^{\alpha}(s)|^p + \E\Big(\sup\limits_{0\leq s\leq t}|I_{1}^{\alpha}(s)|^p\Big) 
\\&\qquad+ \sup\limits_{0\leq s\leq t}|I_{2}^{\alpha}(s)|^p+ \E\Big(\sup\limits_{0\leq s\leq t}|I_{3}^{\alpha}(s)|^p\Big)+ \sup\limits_{0\leq s\leq t}|I_{4}^{\alpha}(s)|^p\Big].
	\end{align*}
	This, together with \eqref{eq1.m}, \eqref{eq1.mm}, \eqref{eq1.m1}, \eqref{eq1.mmm} and \eqref{eq1mm}, we get \eqref{eq2mm}.
\end{proof}
In the following theorem,  we obtain a rate of convergence in $L^p$-distances in the Smoluchowski-Kramers approximation for the displacement process.
	\begin{theorem}
		\label{bd:2.5}
		Let  $\{X^\alpha_t, t\in [0,T]\}$ and $\{X_t, t\in [0,T]\}$ be respectively the solution of \eqref{eqn:2.2} and of \eqref{eqn:2.3} under Assumptions \ref{gt:2.1}. Then, for all $p\geq 2$, $\alpha \geq 1$ and $t\in[0,T]$,
		\begin{equation*}
			\E \Big[\sup_{0\leq s\leq t}|X^{\alpha}_s-X_s|^{p}\Big] \leq C \left[\left( \lambda(t,\alpha(\kappa+\gamma))\right)^{\frac{p}{2}}+\left( \lambda(t,\alpha\kappa)\right)^p \right],
		\end{equation*}
		where $C$ is a positive constant depending on $\{x_0,y_0,\kappa,\gamma,K,L,p,T\}$  but not on $\alpha$ and $t$.
	\end{theorem}
	\begin{proof}
From \eqref{eqn:2.2} and \eqref{eqn:2.3}, we have
		\begin{align*}
X_t^\alpha-X_t & =\dfrac{1}{\kappa+\gamma}\int_{0}^{t}(g(s,X^{\alpha}_s)-g(s,X_s))\ ds - \dfrac{\gamma}{\kappa(\kappa+\gamma)}\int_{0}^{t}( \E[g(t,X^{\alpha}_t)]- \E[g(t,X_t)]) ds \\
&\qquad + \dfrac{1}{\kappa+\gamma}\int_{0}^t(\sigma(s,X^\alpha_s)-\sigma(s,X_s))dW(s) + I_{0}^{\alpha}(t) + I_{1}^{\alpha}(t) - I_{2}^{\alpha}(t) - I_{3}^{\alpha}(t) - I_{4}^{\alpha}(t).
	\end{align*}
Similar to the proof of the previous lemma, by applying Minkowski's inequality \eqref{eq: Minkowski inequality} with $n=8$, H\"older's inequality \eqref{eq: Holder inequality}, the BDG inequality \eqref{eq: BDG inequality2} and Assumptions \ref{gt:2.1}, we get
			\begin{align*}
&\E\Big[\sup\limits_{0\leq s\leq t}|X_s^\alpha-X_s|^p\Big] 
\\&\qquad \leq 8^{p-1}\Bigg[ \dfrac{t^{p-1}}{(\kappa+\gamma)^p}\int_{0}^{t}E|g(s,X^{\alpha}_s)-g(s,X_s)|^p ds + \dfrac{\gamma^pt^{p-1}}{\kappa^p(\kappa+\gamma)^p}\int_{0}^{t}\Big| \E[g(t,X^{\alpha}_t)]- \E[g(t,X_t)]\Big|^p\ ds \\
		&\qquad\qquad+ \dfrac{C_p}{(\kappa+\gamma)^p}\E\left(\int_{0}^t|\sigma(s,X^\alpha_s)-\sigma(s,X_s)|^2ds\right)^{p/2}  + \sup\limits_{0\leq s\leq t}|I_{0}^{\alpha}(s)|^p + \E\Big(\sup\limits_{0\leq s\leq t}|I_{1}^{\alpha}(s)|^p\Big)  
\\&\qquad\qquad\qquad+\sup\limits_{0\leq s\leq t}|I_{2}^{\alpha}(s)|^p+ \E\Big(\sup\limits_{0\leq s\leq t}|I_{3}^{\alpha}(s)|^p\Big) + \sup\limits_{0\leq s\leq t}|I_{4}^{\alpha}(s)|^p\Bigg]
\\&\qquad\leq 8^{p-1}\bigg[ \dfrac{L^pt^{p-1}}{(\kappa+\gamma)^p}\int_{0}^{t}\E|X^{\alpha}_s-X_s|^p ds + \dfrac{L^p\gamma^pt^{p-1}}{\kappa^p(\kappa+\gamma)^p}\int_{0}^{t}\E|X^{\alpha}_s-X_s|^p ds \\
		&\qquad\qquad+ \dfrac{C_pL^pt^{\frac{p}{2}}}{(\kappa+\gamma)^p}\int_{0}^{t}\E|X^{\alpha}_s-X_s|^p ds  + \sup\limits_{0\leq s\leq t}|I_{0}^{\alpha}(s)|^p + \E\Big(\sup\limits_{0\leq s\leq t}|I_{1}^{\alpha}(s)|^p\Big)  
	\\&\qquad\qquad\qquad+\sup\limits_{0\leq s\leq t}|I_{2}^{\alpha}(s)|^p+ \E\Big(\sup\limits_{0\leq s\leq t}|I_{3}^{\alpha}(s)|^p\Big) + \sup\limits_{0\leq s\leq t}|I_{4}^{\alpha}(s)|^p\Bigg].
		\end{align*}
Next we estimate the terms $\E\Big(\sup\limits_{0\leq s\leq t}|I_{i}^{\alpha}(s)|^p\Big)$, $i=1,2,3,4$. We start with $\E\Big(\sup\limits_{0\leq s\leq t}|I_{1}^{\alpha}(s)|^p\Big)$. From definition of $I_1^\alpha(t)$ and Assumptions \ref{gt:2.1} and Minkowski's inequality \eqref{eq: Minkowski inequality} with $n=2$ we obtain that
\begin{align*}|I_{1}^{\alpha}(t)|^{p}&\leq  \dfrac{2^{p}K^p}{(\kappa+\gamma)^p}\left(\int_{0}^{t} \exp{[\alpha(\kappa+\gamma)(s-t)]}(1+|X^{\alpha}_s|)ds\right)^p\\&\leq 
\dfrac{2^{p}K^p}{(\kappa+\gamma)^p}\left(1+\sup_{0\leq t\leq T}|X^{\alpha}_t|^{p}\right)(\lambda(t;(\kappa+\gamma)\alpha))^{p}.
\end{align*} 
		This, together with  the fact that the function $t \mapsto \lambda(t,a)$ is increasing and Lemma \ref{bd:2.4}, implies
	\begin{align}\label{eq1m}\E\Big(\sup\limits_{0\leq s\leq t}|I_{1}^{\alpha}(s)|^p\Big)&\leq \dfrac{2^{p}K^p}{(\kappa+\gamma)^p}\left[1+\E\Big(\sup_{0\leq t\leq T}|X^{\alpha}_t|^{p}\Big)\right](\lambda(t;(\kappa+\gamma)\alpha))^{p}\notag\\&\leq C(\lambda(t;(\kappa+\gamma)\alpha))^{p},
	\end{align}
		where $C$ is a positive constant depending on $\{x_0,y_0,\kappa,\gamma,K,p,T\}$.
	Next, using \eqref{ptme} and Lemma \ref{bd:2.4}, we get
		\begin{equation*}
			\begin{split}
				|I_{2}^{\alpha}(t)| &\leq y_0\lambda(t;\alpha(\kappa+\gamma))+ \dfrac{K}{\kappa+\gamma}\int_{0}^{t}\exp{[\alpha\kappa(s-t)]}[1+\E(|X^{\alpha}_s|)]\ ds\\
				&\qquad+\dfrac{K}{\kappa+\gamma}\int_{0}^{t}\exp{[\alpha(\kappa+\gamma)(s-t)]}[1+\E(|X^{\alpha}_s|)]\ ds\\&
				\leq C\left[\lambda(t;\alpha(\kappa+\gamma)+\int_{0}^{t}\left(\exp{[\alpha\kappa(s-t)]}+\exp{[\alpha(\kappa+\gamma)(s-t)]}\right)ds\right] \\&
				\leq C\left[\lambda(t;\alpha(\kappa+\gamma)+\lambda(t,\alpha\kappa) \right],
					\end{split}
		\end{equation*}
		where $C$ is constant depending on $\{x_0,y_0,\kappa,\gamma,K,p\}$. Thus, \begin{equation}\label{eq2m}\sup\limits_{0\leq s\leq t}|I_{2}^{\alpha}(s)|^p \leq C \left[\left( \lambda(t,\alpha(\kappa+\gamma))\right)^p+\left( \lambda(t,\alpha\kappa)\right)^p \right]. 	\end{equation}
		
Applying the BDG inequality \eqref{eq: BDG}, H\"older's inequality \eqref{eq: Holder inequality} and Lemma \ref{bd:2.4},  one can derive that 
		\begin{align*}
&\E\Big(\sup\limits_{0\leq s\leq t}|I_{3}^{\alpha}(s)|^p\Big) \leq \dfrac{C_p}{(\kappa+\gamma)^p}E\left(\int_{0}^t\exp{[2\alpha(\kappa+\gamma)(s-t)]}| \sigma(s,X^\alpha_s)|^2ds\right)^{p/2}
\\&\quad\leq \dfrac{C_p}{(\kappa+\gamma)^p}\left(\int_{0}^t\exp{\left[\frac{p}{p-1}\alpha(\kappa+\gamma)(s-t)\right]}ds\right)^{p/2-1}\int_{0}^t\exp{\left[\frac{p}{2}\alpha(\kappa+\gamma)(s-t)\right]}\E|\sigma(s,X^\alpha_s)|^p\,ds
			\\&\quad\leq \dfrac{2^{p-1}C_p}{(\kappa+\gamma)^p}\left(\int_{0}^t\exp{\left[\frac{p}{p-1}\alpha(\kappa+\gamma)(s-t)\right]}ds\right)^{p/2-1}\int_{0}^t\exp{\left[\frac{p}{2}\alpha(\kappa+\gamma)(s-t)\right]}(1+\E|X^\alpha_s|^p)ds\\&\quad\leq
			C\left(\lambda(t,\frac{p}{p-1}\alpha(\kappa+\gamma))\right)^{p/2-1}\lambda(t,\frac{p}{2}\alpha(\kappa+\gamma)),
		\end{align*}
		where $C$ is constant depending on $\{x_0,y_0,\kappa,\gamma,K,p,T\}$.	On the other hand, for all $t>0$ and $a>0$, we have
$$
\frac{\partial \lambda(t,a)}{\partial a}=\dfrac{(1+at)e^{-at}-1}{a^2}< 0.
$$
Thus the function $a \mapsto \lambda(t,a)$ is decreasing.  Hence we get
		$$\E\Big(\sup\limits_{0\leq s\leq t}|I_{3}^{\alpha}(s)|^p\Big)
		\leq C\left( \lambda(t,\alpha(\kappa+\gamma))\right)^{p/2}.$$
		Now we consider $\sup\limits_{0\leq s\leq t}|I_{4}^{\alpha}(s)|^{p}$. Using Lemma \ref{bd:2.4} we can derive that
		\begin{equation*}
			\begin{split}
			\sup\limits_{0\leq s\leq t}|I_{4}^{\alpha}(s)| \leq \dfrac{\gamma}{\kappa+\gamma}\left[y_0\lambda(t;\alpha\kappa)+\dfrac{K}{\kappa}\int_{0}^{t}\exp{[\alpha(\kappa+\gamma)(u-t)]}\left[1+\E(|X^{\alpha}_u|)\right]\,du\right].
			\end{split}
		\end{equation*}
		 Thus, $$\sup\limits_{0\leq s\leq t}|I_{4}^{\alpha}(s)|^{p} \leq C \left[\left( \lambda(t,\alpha(\kappa+\gamma))\right)^p+\left( \lambda(t,\alpha\kappa)\right)^p \right],$$
		where $C$ is constant depending on $\{x_0,y_0,\kappa,\gamma,K,L,p\}$.
		From the above estimates, together with the fact that $\lambda(t,\alpha(\kappa+\gamma))\leq t\leq T,$ one sees that
		\begin{equation*}
			\begin{split}
				\E\Big[\sup\limits_{0\leq s\leq t}|X_s^\alpha-X_s|^p\Big] &\leq
				C\left[t^{2p-1}\int_{0}^{t}\E|X^{\alpha}_s-X_s|^{p}ds+\sum_{i=0}^{4}\E\Big(\sup\limits_{0\leq s\leq t}|I_{i}^{\alpha}(s)|^p\Big)\right]\\
				&\leq C\left[\int_{0}^{t}\E\Big[\sup\limits_{0\leq u\leq s}|X_u^\alpha-X_u|^p\Big]ds+\left( \lambda(t,\alpha(\kappa+\gamma))\right)^{\frac{p}{2}}+\left( \lambda(t,\alpha\kappa)\right)^p\right],\\
			\end{split}
		\end{equation*}
		where $C$ is constant depending only on $\{x_0,y_0,\kappa,\gamma,K,L,p\}$. Using Growwall's inequality, we obtain the claimed estimate and complete the proof.
	\end{proof}
In the following lemma, we show Malliavin differentiability of $X^\alpha_t$ and $X_t$.
	\begin{lemma}
		\label{bd:2.3}
		Under Assumptions \ref{gt:2.1}, the solutions $\{X^{\alpha}_t, t\in [0,T]\}$ and $\{X_t, t\in [0,T]\}$ of (\ref{eqn:2.2}) and (\ref{eqn:2.3}) respectively are Malliavin differentiable random variables. Moreover, the derivatives $D_rX^{\alpha}_t, D_rX_t$ satisfy $D_rX^{\alpha}_t= D_rX_t = 0$ for $r \geq t$ and for $0\leq r < t\leq T,$
		
		\begin{align}\label{eqn:2.4}
		D_rX^{\alpha}_t &=\dfrac{\sigma(r,X_r^\alpha)}{\kappa+\gamma}(\exp{[\alpha(\kappa+\gamma)(r-t)]}-1)-\dfrac{1}{\kappa+\gamma}\int_r^t\bar{g}^\alpha(s) D_rX_s^\alpha ds\notag
\\&\qquad+ \dfrac{1}{\kappa+\gamma}\int_r^t\bar{\sigma}^\alpha(s) D_rX_s^\alpha\, dW_s+\dfrac{1}{\kappa+\gamma}\int_{r}^{t}\exp{[\alpha(\kappa+\gamma)(s-t)]}\bar{g}^\alpha(s)D_rX_s^\alpha ds\notag
\\&\qquad+\dfrac{1}{\kappa+\gamma}\int_{r}^{t}\exp{[\alpha(\kappa+\gamma)(s-t)]}\bar{\sigma}^\alpha(s)D_rX_s^\alpha\, dW_s,
\\\notag\\		
D_rX_t &= \dfrac{\sigma(r,X_r)}{\kappa+\gamma}-\dfrac{1}{\kappa+\gamma}\int_r^t\bar{g}(s) D_rX_s ds+ \dfrac{1}{\kappa+\gamma}\int_r^t\bar{\sigma}(s) D_rX_s \, dW_s,\label{eqn:2.4m}
		\end{align}
		where $\bar{g}(s), \bar{g}^\alpha(s),\bar{\sigma}(s), \bar{\sigma}^\alpha(s)$ are adapted stochastic processes and are bounded by $L$.
	\end{lemma}
	\begin{proof}
		From the second equation of \eqref{eq1}  we have $$	Y^{\alpha}_{t} =y_0-\alpha(\kappa+\gamma)\int_0^tY^\alpha_sds- \alpha\int_0^tg(s,X^\alpha_{s})ds-\alpha\gamma \int_0^t\E(Y^\alpha_s)\,ds+\alpha \int_0^t\sigma(s,X^\alpha_{s})\,dW_{s}.$$
		Using  Minkowski's inequality \eqref{eq: Minkowski inequality} with $n=5$, H\"older's inequality \eqref{eq: Holder inequality} with $p=2$, the BDG inequality \eqref{eq: BDG}, Assumptions \ref{gt:2.1} and Lemma \ref{bd:2.4}  we get 
		 	\begin{align*}		\E(\sup\limits_{0\leq s\leq t}|Y^{\alpha}_s|^{2}) &\leq 5\bigg(y^2_0+t\alpha^2(\kappa+\gamma)^2\int_0^t\E|Y^\alpha_s|^2ds+ t\alpha^2\int_0^tg^2(s,X^\alpha_{s})\,ds-t\alpha^2\gamma^2 \int_0^t\E|Y^\alpha_s|^2\,ds
\\&\qquad\qquad+\alpha^2 \int_0^t\sigma^2(s,X^\alpha_{s})\,ds\bigg)\\&
		\leq C\left(1+\int_0^t\E(\sup\limits_{0\leq u\leq s}|Y^\alpha_u|^2)\,ds+ \int_0^t(1+\E|X^\alpha_{s}|^2)\,ds\right)
		\\&\leq C\left(1+\int_0^t\E(\sup\limits_{0\leq u\leq s}|Y^\alpha_u|^2)\,ds\right).
		 \end{align*}
		 Applying Gronwall's lemma we obtain
		 \begin{equation*}
		 \E\Big[\sup_{0\leq t\leq T}|Y^{\alpha}_t|^{2}\Big] \leq C.
		 \end{equation*}
	This, together with Lemma \ref{bd:2.4} we deduce that	 $\E[\sup\limits_{0\leq t \leq T}|Y^{\alpha}_t|]$ and $\E[\sup\limits_{0\leq t \leq T}|X^{\alpha}_t|]$  are bounded. Then, by  Assumptions \ref{gt:2.1} and the dominated convergence theorem,  the integrals $I_0^{\alpha}(t), I_2^{\alpha}(t)$ and $ I_4^{\alpha}(t)$ are  continuous functions.
		
Let us define 
$$
f(t):= x_0-\dfrac{\gamma}{\kappa+\gamma}\int_{0}^{t}G^{\alpha}(u)\ du + I_0^{\alpha}(t) -  I_2^{\alpha}(t) - I_4^{\alpha}(t).
$$
Then $f(t)$ is continuous function in $[0,T]$ and Equation \eqref{eqn:2.2} becomes
	\begin{equation}
	\label{eqn:2.5}
	\begin{split}
	X^{\alpha}_t &= f(t) +  \dfrac{1}{\kappa+\gamma}\int_{0}^{t}\left(\exp{[\alpha(\kappa+\gamma)(s-t)]}-1\right)g(s,X^{\alpha}_s)\ ds\\
	&\qquad+ \dfrac{1}{\kappa+\gamma}\int_{0}^{t} (\exp{[\alpha(\kappa+\gamma)(s-t)]}-1)\sigma(s,X_s)\ dW_s.
	\end{split}
	\end{equation}
	Now, we consider the Picard approximation sequence $\{X^{\alpha,n}_t,  t\in [0,T]\}_{n\geq 0}$ given by
		\begin{equation*}
		\begin{cases}
		X^{\alpha,0}_t = f(t),\\
		\begin{split}
		X^{\alpha,n+1}_t &= f(t) +  \dfrac{1}{\kappa+\gamma}\int_{0}^{t}\left(\exp{[\alpha(\kappa+\gamma)(s-t)]}-1\right)g(s,X^{\alpha,n}_s)\ ds\\
		&\quad +\dfrac{\delta}{\kappa+\gamma}\int_{0}^{t} (\exp{[\alpha(\kappa+\gamma)(s-t)]}-1)\sigma(s,X_s^{\alpha,n})\ dW_s,  \,t\in [0,T],\, n\geq 0.
		\end{split}
		\end{cases}
		\end{equation*}
		From this, using  the same method as in the proof of \cite[Theorem 2.2.1]{nualartm2}, we conclude that the solution $\{X^{\alpha}_t, t\in [0,T]\}$ of (\ref{eqn:2.5}) (thus, of (\ref{eqn:2.2})) is Malliavin's differentable.  Obviously, the solution $\{X^{\alpha}_t, t\in [0,T]\}$ is $\mathbb{F}$-adapted. Hence, we always have $D_\theta X^\alpha_t=0$ for $\theta >t.$ For $\theta\leq t,$  from \cite[Proposition 1.2.4]{nualartm2} and Lipschitz property of $g$ and $\sigma$, there exist adapted  processes $\bar{g}^\alpha(s)$, $\bar{\sigma}^\alpha(s)$ uniformly bounded by $L$ such that  $D_\theta g(s,X^\alpha_s)=\bar{g}^\alpha(s)D_\theta X^\alpha_s$ and   $D_\theta \sigma(s,X^\alpha_s)=\bar{\sigma}^\alpha(s)D_\theta X^\alpha_s$. Then we obtain \eqref{eqn:2.4} by applying  the operator $D$ to the equation \eqref{eqn:2.5}.
		
The proof for the solution $X_t$ of (\ref{eqn:2.3}) is similar. 
	\end{proof}
	
	\begin{rem}\label{rem1}
		If $g$ and $ \sigma$ are continuously differentiable, then $\bar{g}(s)=g'_2(s,X_s)$, $\bar{g}^\alpha(s)= g'_2(s,X_s^\alpha), \bar{\sigma}(s)= \sigma'_2(s,X_s)$ and $ \bar{\sigma}^\alpha(s)= \sigma'_2(s,X_s^\alpha)$. Here, for a function $h(t,x),$ we use the convention $h'_2(t,x)=\frac{\partial h(t,x)}{\partial x}.$
	\end{rem}
In the next lemma, we show that the moments of the Malliavin's derivative of solutions of \eqref{eqn:2.3} are bounded.	
	\begin{lemma}
		\label{bd:2.6}
		Let  $\{X_t, t\in [0,T]\}$ be the solution of \eqref{eqn:2.3} with Assumptions \ref{gt:2.1}. Then, for all $p \geq 2$,	we have
		\label{lm:2.5}
		\begin{equation*}
			\sup_{0\leq t,r\leq T}\E(|D_rX_t|^p) < \infty.
		\end{equation*}
	\end{lemma}
	\begin{proof}
Using Minkowski's inequality \eqref{eq: Minkowski inequality} with $n=3$, H\"older's inequality \eqref{eq: Holder inequality}, the BDG inequality \eqref{eq: BDG}, Assumptions \ref{gt:2.1} and Lemma \ref{bd:2.4}, and noting that $|\bar{g}(s)|\leq L, |\bar{\sigma}(s)|\leq L$, it follows from \eqref{eqn:2.4m} that
	\begin{align*}	
&\E[|D_rX_t|^p] 
\\&\quad\leq 3^{p-1}\left(\dfrac{\E[|\sigma(r,X_r)|^p]}{(\kappa+\gamma)^p}-\dfrac{1}{(\kappa+\gamma)^p}\E\left(\int_r^t\bar{g}(s) D_rX_s\, ds\right)^p+ \dfrac{1}{(\kappa+\gamma)^p}\E\left(\int_r^t|\bar{\sigma}(s) D_rX_s|^2\, ds\right)^{p/2}\right)
\\&\quad\leq 3^{p-1}\left(\dfrac{K^p(1+\E[|X_t|^p])}{(\kappa+\gamma)^p}-\dfrac{L^p(t-r)^{p-1}}{(\kappa+\gamma)^p}\int_r^t \E[|D_rX_s|^p] \,ds+ \dfrac{L^p(t-r)^{p/2-1}}{(\kappa+\gamma)^p}\int_r^t \E[|D_rX_s|^p] ds\right)
\\&\quad
\leq C\left[1+\int_{r}^{t}|D_rX^{\alpha}(s)|^p \,ds\right],
\end{align*}
		where $C$ is a positive constant depending only on $\{\kappa,\gamma,K,L,p,T\}$.
		
		Taking the expectation and using Gronwall's inequality, we obtain the claimed estimate.
	\end{proof}
The following lemma provides an upper bound for the difference between the derivatives of the solutions	of \eqref{eqn:2.2} and \eqref{eqn:2.3}.
	\begin{lemma}
		\label{bd:2.7}
		Let  $\{X^\alpha_t, t\in [0,T]\}$ and $\{X_t, t\in [0,T]\}$ be respectively the solution of \eqref{eqn:2.2} and of \eqref{eqn:2.3} under Assumptions \ref{gt:2.1} and \ref{gt:2.2}. Then, for all $\alpha \geq 1$,
		\begin{equation*}
			\E\left[\|DX^{\alpha}_t-DX_t\|_{\mathcal{H}}^2 \right]\leq C(\lambda(t,\alpha(\kappa+\gamma))+\left( \lambda(t,\alpha\kappa)\right)^2),
		\end{equation*}
		where $C$ is a positive constant depending only on $\{x_0,y_0,\kappa,\gamma,K,L,M,T\}$.
	\end{lemma}
	\begin{proof}
Under Assumptions \ref{gt:2.2}, $g$ and $\sigma$ are twice differentiable, thus (see Remark \ref{rem1}) $\bar{g}(s)=g'_2(s,X_s)$, $\bar{g}^\alpha(s)= g'_2(s,X_s^\alpha), \bar{\sigma}(s)= \sigma'_2(s,X_s)$ and $ \bar{\sigma}^\alpha(s)= \sigma'_2(s,X_s^\alpha)$. Furthermore,
\begin{equation}
\label{eq:boundderivative}
|g'_2(s,X_s^\alpha)-g'_2(s,X_s)|\leq M|X_s^\alpha-X_s|,\quad |\sigma'_2(s,X_s^\alpha)-\sigma'_2(s,X_s|\leq M|X_s^\alpha-X_s|.
\end{equation}
From \eqref{eqn:2.4} and \eqref{eqn:2.4m} we have
		\begin{align}\label{eql3.4.1}
	D_rX^{\alpha}_t -D_rX_t &=\left(\dfrac{\sigma(r,X_r^\alpha)}{\kappa+\gamma}(\exp{[\alpha(\kappa+\gamma)(r-t)]}-1)-\dfrac{\sigma(r,X_r)}{\kappa+\gamma}\right)\notag
	\\&\qquad-\dfrac{1}{\kappa+\gamma}\int_r^t\left(g'_2(s,X_s^\alpha) D_rX_s^\alpha-g'_2(s,X_s) D_rX_s\right)\, ds\notag\\&\qquad+ \dfrac{1}{\kappa+\gamma}\int_r^t\left(\sigma'_2(s,X_s^\alpha) D_rX_s^\alpha-\sigma'_2(s,X_s) D_rX_s\right)\, dW_s\notag\\
	&\qquad+\dfrac{1}{\kappa+\gamma}\int_{r}^{t}\exp{[\alpha(\kappa+\gamma)(s-t)]}g'_2(s,X_s^\alpha)D_rX^{\alpha}_s \,ds\notag\\
	&\qquad+\dfrac{1}{\kappa+\gamma}\int_{r}^{t}\exp{[\alpha(\kappa+\gamma)(s-t)]}\sigma'_2(s,X_s)D_rX^{\alpha}_s\, dW_s.
	\end{align}
		Now, we shall estimate each term in the right hand side of \eqref{eql3.4.1}. First, using  Assumptions \ref{gt:2.1}, Lemma \ref{bd:2.4} and Theorem \ref{bd:2.5} for $p=2$, we can derive that 
		\begin{align*}
&\E\left(\dfrac{\sigma(r,X_r^\alpha)}{\kappa+\gamma}(\exp{[\alpha(\kappa+\gamma)(r-t)]}-1)-\dfrac{\sigma(r,X_r)}{\kappa+\gamma}\right)^2 
\\&\qquad\leq 2\left[\dfrac{\E\big[|\sigma(r,X_r^\alpha)-\sigma(r,X_r)|^2\big]}{(\kappa+\gamma)^2}+ \dfrac{\E\big[|\sigma(r,X_r^\alpha)|^2\big]}{(\kappa+\gamma)^2}\exp{[2\alpha(\kappa+\gamma)(r-t)]}\right]\\&
\qquad\leq 2\left[\dfrac{L^2\E|X_r^\alpha-X_r|^2}{(\kappa+\gamma)^2}+ \dfrac{2K^2(1+\E[|X_r^\alpha|^2])}{(\kappa+\gamma)^2}\exp{[2\alpha(\kappa+\gamma)(r-t)]}\right]
			\\
			&\qquad\leq C \left[ \lambda(t,\alpha(\kappa+\gamma))+\left( \lambda(t,\alpha\kappa)\right)^2 +\exp{[2\alpha(\kappa+\gamma)(r-t)]}\right],
		\end{align*}
		where $C$ is constant depending only on $\{x_0,y_0,\kappa,\gamma,K,L,T\}$.
	From  H\"older's inequality, Assumptions \ref{gt:2.1}, Lemma \ref{bd:2.5}, Lemma \ref{bd:2.6}  and Theorem \ref{bd:2.5} for $p=4$, together with \eqref{eq:boundderivative}  we get 
		\begin{align*}
{}&	\E\left(\int_r^t\left(g'_2(s,X_s^\alpha) D_rX_s^\alpha-g'_2(s,X_s) D_rX_s\right)\, ds\right)^2 
\\&\qquad\leq 2	\E\left(\int_r^t\left(g'_2(s,X_s^\alpha) -g'_2(s,X_s)\right) D_rX_s\, ds\right)^2+2\E\left(\int_r^tg'_2(s,X_s^\alpha)\left(D_rX_s^\alpha- D_rX_s\right)\, ds\right)^2
\\&\qquad\leq 2(t-r)\Bigg[\int_r^t\E|\left(g'_2(s,X_s^\alpha) -g'_2(s,X_s)\right) D_rX_s|^2\, ds
\\&\qquad\qquad\qquad+\int_r^t\E\left|g'_2(s,X_s^\alpha)\left(D_rX_s^\alpha- D_rX_s\right)\right|^2\, ds\Bigg]
	\\&\qquad\leq 2\M^2(t-r)\int_r^t\left(\E\left[|X_s^\alpha -X_s|^4\right]\right)^{1/2} \left(\E|D_rX_s|^4\right)^{1/2}\, ds
\\&\qquad\qquad\qquad+2L^2	(t-r)\int_r^t\E\left[|D_rX_s^\alpha- D_rX_s|^2\right]\, ds
	\\
	&\qquad\leq C \left(\lambda(t,\alpha(\kappa+\gamma))+\left( \lambda(t,\alpha\kappa)\right)^2 +\int_r^t\E\left[|D_rX_s^\alpha- D_rX_s|^2\right]\, ds\right).
	\end{align*}
By It\^o's isometry formula, H\"older's inequality, Assumptions \ref{gt:2.1}, Lemma \ref{bd:2.5}, Lemma \ref{bd:2.6} and Theorem \ref{bd:2.5} for $p=4$, together with \eqref{eq:boundderivative},  we have
		\begin{align*}
	{}&	\E\left(\int_r^t\left(\sigma'_2(s,X_s^\alpha) D_rX_s^\alpha-\sigma'_2(s,X_s) D_rX_s\right)\, dW_s\right)^2 
	\\&\qquad\leq 2	\E\left(\int_r^t\left(\sigma'_2(s,X_s^\alpha) -\sigma'_2(s,X_s)\right) D_rX_s\, dW_s\right)^2+2	\E\left(\int_r^t\sigma'_2(s,X_s^\alpha)\left(D_rX_s^\alpha- D_rX_s\right)\, dW_s\right)^2
	\\&\qquad\leq 2\int_r^t \E|\left(\sigma'_2(s,X_s^\alpha) -\sigma'_2(s,X_s)\right) D_rX_s|^2\, ds+2\int_r^t\E|\sigma'_2(s,X_s^\alpha)\left(D_rX_s^\alpha- D_rX_s\right)|^2\, ds
	\\
	&\qquad\leq 2\M^2\int_r^t\left(\E|X_s^\alpha -X_s|^4\right)^{1/2} \left(\E|D_rX_s|^4\right)^{1/2} ds+2L^2	\int_r^t\E|D_rX_s^\alpha- D_rX_s|^2 ds
	\\
	&\qquad\leq C \left[\lambda(t,\alpha(\kappa+\gamma))+\left( \lambda(t,\alpha\kappa)\right)^2 +\int_r^t\E\big[|D_rX_s^\alpha- D_rX_s|^2\big] ds\right].
	\end{align*}
Next, using  H\"older's inequality, Assumptions \ref{gt:2.1} and Lemma \ref{bd:2.6}  one sees that
		\begin{align*}
	&\E\left(\int_{r}^{t}\exp{[\alpha(\kappa+\gamma)(s-t)]}g'_2(s,X_s^\alpha)D_rX^{\alpha}_s ds\right)^2 
	\\&\qquad\leq (t-r)L^2\int_{r}^{t}\exp{[2\alpha(\kappa+\gamma)(s-t)]}\E\big[|D_rX^{\alpha}_s|^2\big] ds
	\\&\qquad\leq C\int_{r}^{t}\exp{[2\alpha(\kappa+\gamma)(s-t)]} ds
	\\&\qquad\leq C\lambda(t,\alpha(\kappa+\gamma)).
\end{align*}
	By the same estimate for the last term in the right hand side of \eqref{eql3.4.1}, we can obtain
		\begin{align*}
&\E\left(\int_{r}^{t}\exp{[\alpha(\kappa+\gamma)(s-t)]}\sigma'_2(s,X_s)D_rX^{\alpha}_s dW_s\right)^2 
\\&\qquad\leq L^2\int_{r}^{t}\exp{[2\alpha(\kappa+\gamma)(s-t)]}E\big[|D_rX^{\alpha}_s|^2\big]\, ds
\\&\qquad\leq C\int_{r}^{t}\exp{[2\alpha(\kappa+\gamma)(s-t)]} ds\\&\qquad\leq C\lambda(t,\alpha(\kappa+\gamma)).
	\end{align*}
		From the above estimates, together with the fact that the function $a \mapsto \lambda(t,a)$ is decreasing one can derive that
		\begin{equation*}
			\begin{split}
				&\int_{0}^{t}\E\left[|D_rX^{\alpha}_t-D_rX_t|^2\right]\ dr\\ 
				&\leq C\,\int_{0}^{t}\left(\lambda(t,\alpha(\kappa+\gamma))+\left( \lambda(t,\alpha\kappa)\right)^2+\exp{[2\alpha(\kappa+\gamma)(r-t)]}+\int_{r}^{t}\E\left[|D_rX^{\alpha}_u-D_rX_u|^2\right]\ du\right)\,dr\\
				&\leq C\,\left[\lambda(t,\alpha(\kappa+\gamma))+\left( \lambda(t,\alpha\kappa)\right)^2+\lambda(t,2\alpha(\kappa+\gamma))+\int_{0}^{t}dr\int_{r}^{t}\E\left[|D_rX_u-D_rX^{\alpha}_u|^2\right]\,du\right]\\
				&\leq C\,\left[\lambda(t,\alpha(\kappa+\gamma))+\left( \lambda(t,\alpha\kappa)\right)^2+\int_{0}^{t}du\int_{0}^{u}\E\left[|D_rX_u-D_rX^{\alpha}_u|^2\right]\,dr\right],
			\end{split}
		\end{equation*}
		where $C$ is constant depending only on $\{x_0,y_0,\kappa,\gamma,K,L,\M, T\}$.
		
		Let $\phi(t):= \displaystyle\int_{0}^{t}\E\left[|D_rX_t-D_rX^{\alpha}_t|^2\right]dr$, then we have 
		$$
		\phi(t) \leq C\left[\lambda(t,\alpha(\kappa+\gamma))+\left( \lambda(t,\alpha\kappa)\right)^2+\displaystyle\int_{0}^{t}\phi(u)\ du\right].
		$$
		Thus, applying Gronwall's inequality, we get 
		\begin{equation*}
			\phi(t) \leq C(\lambda(t,\alpha(\kappa+\gamma))+\left( \lambda(t,\alpha\kappa)\right)^2)\exp{(Ct)},
		\end{equation*}
		where $C$ is constant depending only on $\{x_0,y_0,\kappa,\gamma,K,L,\M,T\}$. This completes the proof of the lemma. 
	\end{proof}
Now, we give explicit bounds on the total variation distance between the solution $X^\alpha(t)$ of \eqref{eqn:2.2}
and the solution $X_t$ of \eqref{eqn:2.3}.
	\begin{theorem}
		\label{dl:2.1}
		Let $\{X^\alpha_t, t\in [0,T]\}$ and $\{X_t, t\in [0,T]\}$ be, respectively, the solution of \eqref{eqn:2.2}and of \eqref{eqn:2.3}  with Assumptions \ref{gt:2.1} and \ref{gt:2.2}. We further assume that $|\sigma(t,x)|\geq \sigma_0>0$ for all $(t,x)\in [0,T]\times \mathbb{R}.$  Then, for each $\alpha\geq 1$ and $t\in (0,T],$ 
		\begin{equation*}
			d_{TV}(X^{\alpha}_t,X_t) \leq C\sqrt{t^{-1}(\lambda(t,\alpha(\kappa+\gamma))+\left( \lambda(t,\alpha\kappa)\right)^2)},
		\end{equation*}
		where $C$ is a constant depending only on $\{x_0,y_0,\sigma_0, \kappa,\gamma,K,L,\,M,T\}$.
		
	\end{theorem}
	\begin{proof}
		Lemma \ref{bd:2.1} gives us
		\begin{equation*}
			d_{TV}(X^{\alpha}_t,X_t) \leq \|X^{\alpha}_t-X_t\|_{1,2} \left[3\left(\E\|D^2X_t\|^4_{\mathcal{H}\bigotimes\mathcal{H}} \right)^{1/4}\left( \E\|DX_t\|_\mathcal{H}^{-8}\right)^{1/4}+ 2\left(\E\|DX_t\|_\mathcal{H}^{-2}\right)^{1/2}\right].
		\end{equation*}
		Thanks to Theorem \ref{bd:2.5} and Lemma \ref{bd:2.7}, we obtain
		\begin{align}
			\label{eqn:3.1}
			d_{TV}(X^{\alpha}_t,X_t) &\leq C\sqrt{(\lambda(t,\alpha(\kappa+\gamma))+\left( \lambda(t,\alpha\kappa)\right)^2)}\notag\\&\times\left[3\left(\E\|D^2X_t\|^4_{\mathcal{H}\bigotimes\mathcal{H}} \right)^{1/4}\left( \E\|DX_t\|_\mathcal{H}^{-8}\right)^{1/4}+ 2\left(\E\|DX_t\|_\mathcal{H}^{-2}\right)^{1/2}\right],
		\end{align}
		where $C$ is a constant depending only on $\{x_0,y_0,\kappa,\gamma,K,L,\M,T\}$.
		Now, from \eqref{eqn:2.4m}, one sees that
		$$
		D_\theta X_t=\dfrac{\sigma(\theta,X_\theta)}{\kappa+\gamma}
		\exp\bigg(\int_{\theta}^{t}\bigg(\dfrac{g_2'(s,X_s)}{\kappa+\gamma}-\frac{1}{2(\kappa+\gamma)^2}\sigma_2'(s,X_s)^2\bigg)ds +\dfrac{1}{\kappa+\gamma}\int_{\theta}^{t}\sigma_2'(s,X_s)dB_s\bigg),
		$$
		which implies that
		$$|D_\theta X_t|^2\geq \dfrac{\sigma_0^2}{(\kappa+\gamma)^2}e^{-(\frac{2L}{\kappa+\gamma}+\frac{L^2}{(\kappa+\gamma)^2})T}\exp\bigg(\dfrac{2}{\kappa+\gamma}\int_{\theta}^{t}\sigma_2'(s,X_s)dB_s\bigg),\quad  0\leq \theta\leq t\leq T. $$
		Define the stochastic process $M_t:=\int_{0}^{t}\sigma_2'(s,X_s)dB_s,0\leq t\leq T,$ one derives that
		\begin{align*}
		\|DX_t\|^2_{\mathcal{H}}\geq \int\limits_{0}^t|D_\theta X_t|^2d\theta& \geq \dfrac{\sigma_0^2}{(\kappa+\gamma)^2}e^{-(\frac{2L}{\kappa+\gamma}+\frac{L^2}{(\kappa+\gamma)^2})T}e^{\frac{2M_t}{\kappa+\gamma}}\int_0^te^{\frac{2M_\theta}{\kappa+\gamma}}d\theta\\&\geq t\dfrac{\sigma_0^2}{(\kappa+\gamma)^2}e^{-(\frac{2L}{\kappa+\gamma}+\frac{L^2}{(\kappa+\gamma)^2})T}e^{\frac{4}{\kappa+\gamma} \min\limits_{0\leq t\leq T}M_t},
		\end{align*}
for $0\leq t\leq T$. We observe that $M_t$ is a martingale with bounded quadratic variation. Indeed, $\langle M\rangle_t=\int_{0}^{t}\sigma_2'(s,X_s)^2ds\leq L^2T.$ So, by Dubins and Schwarz's theorem (see, e.g. Theorem 3.4.6 in \cite{Karatzas1991}) there exists a  Wiener process $(\widehat{W}_t)_{t\geq 0}$ such that $M_t=\widehat{W}_{\langle M\rangle_t}.$ Then, we arrive at the following
		\begin{align*}	\|DX_t\|^2_{\mathcal{H}}\geq t\dfrac{\sigma_0^2}{(\kappa+\gamma)^2}e^{-(\frac{2L}{\kappa+\gamma}+\frac{L^2}{(\kappa+\gamma)^2})T}e^{\frac{4}{\kappa+\gamma} \min\limits_{0\leq t\leq L^2T}\widehat{W}_t} ,\,\,0\leq t\leq T.
		\end{align*}
		This implies that, for each $r\geq 2$ and  $0< t\leq T,$ \begin{align*}\label{ptm4}\E\left(\|DX_t\|_\mathcal{H}^{-r}\right)\leq  t^{-r/2}\dfrac{(\kappa+\gamma)^r}{\sigma_0^r}e^{(\frac{rL}{\kappa+\gamma}+\frac{rL^2}{2(\kappa+\gamma)^2})T}\E\left[e^{\frac{2r}{\kappa+\gamma}\max\limits_{0\leq t\leq L^2T}(-\widehat{W}_t)}\right].
		\end{align*}
		On the other hand, by  Fernique's theorem, we always have $$\E\left[e^{\frac{2r}{\kappa+\gamma}\max\limits_{0\leq t\leq L^2T}(-\widehat{W}_t)}\right]<\infty.$$
	Hence
	\begin{equation}
\label{ptm4}\E\left(\|DX_t\|_\mathcal{H}^{-r}\right)\leq  Ct^{-r/2}.
	\end{equation}
	We observe that, from \eqref{eqn:2.4}, for $\gamma,\theta\leq t,$ under Assumptions \ref{gt:2.2},
		\begin{align*}
		D_{\gamma}D_{\theta}X_{t}
		&=\sigma'(\theta, X_\theta)D_\gamma X_\theta+\sigma'(\gamma,X_\theta)D_\theta X_\gamma \notag\\&
		\qquad+\int_{\theta\vee \gamma}^t\left(g''(s,X_s) D_\theta X_sD_\gamma X_s+g'(s,X_s)D_\gamma D_\theta X_s\right) ds\notag\\&
		\qquad+\int_{\theta\vee \gamma}^t\left(\sigma''(s,X_s)
		D_\theta X_sD_\gamma X_s+\sigma' (s,X_s)D_\gamma D_\theta X_s\right) dW_s,
		\end{align*}
		Now, using Minkowski's inequality \eqref{eq: Minkowski inequality} with $n=4$, H\"older's inequality \eqref{eq: Holder inequality}, the BDG inequality \eqref{eq: BDG}, Assumptions \ref{gt:2.1} and \ref{gt:2.2}, we can deduce
			\begin{align*}
	\E|D_{\gamma}D_{\theta}X_{t}|^4
		&\leq 64\Bigg[L^4\E|D_\gamma X_\theta|^4+L^4\E|D_\theta X_\gamma|^4 \notag\\&
		\qquad+8(t^3+C_4t)\int_{\theta\vee \gamma}^t\left(M^4 (\E|D_\theta X_s|^2)^2(\E|D_\gamma X_s|^2)^2+L^4\E|D_\gamma D_\theta X_s|^4\right) ds\Bigg].
		\end{align*}
		This, with together Lemma \ref{bd:2.6}, gives us 
		\begin{align*}
		\E\left[|D_{\gamma}D_{\theta}X_{t}|^{4}\right]\leq C+C\int_{\theta\vee \gamma}^t\E\big[|D_\gamma D_\theta X_s|^{4}\big]\,ds,
		\end{align*}
		where $C$ is a positive constant.	By Gronwall's inequality, we can verify that
		$$	\E\left[|D_{\gamma}D_{\theta}X_{t}|^{4}\right]\leq Ce^{C(t-\theta\vee \gamma)}\leq C \ \ \forall \ 0\leq \theta, \gamma\leq t\leq T.$$ 
		Therefore,
		\begin{align}\label{ptm3}\E\|D^2X_t\|^4_{\mathcal{H}\bigotimes\mathcal{H}}&\leq t^{2}\int_0^t\int_0^t\E\big[|D_{\gamma}D_{\theta}X_{t}|^{4}\big]\,d\theta\, d\gamma\leq t^{2}\int_0^t\int_0^tCd\theta d\gamma = Ct^{4},
		\end{align}
		where $C$ is a constant depending only on $\{x_0,y_0,\kappa,\gamma,K,L,T\}$. 
		
	Combining  \eqref{eqn:3.1},   \eqref{ptm4} and \eqref{ptm3}, we can conclude that 
		\begin{align*}
			d_{TV}(X^{\alpha}_t,X_t) \leq & C\sqrt{(\lambda(t,\alpha(\kappa+\gamma))+\left( \lambda(t,\alpha\kappa)\right)^2)}\times\left[C+ Ct^{-1/2}\right]\\&\leq
			C\sqrt{t^{-1}(\lambda(t,\alpha(\kappa+\gamma))+\left( \lambda(t,\alpha\kappa)\right)^2)}, 
		\end{align*}
		where $C$ is a constant depending only on $\{x_0,y_0,\sigma_0, \kappa,\gamma,K,L,\,M,T\}$. This completes the proof. 
	\end{proof}
From Theorem \ref{dl:2.1}, 	together with the fact that  for all $t>0$ and $a>0$, $\lambda(t,a)<\dfrac{1}{a}$, then we get the following Corollary.
\begin{corollary}
	\label{cor:2.1}
	Let $\{X^\alpha_t, t\in [0,T]\}$ and $\{X_t, t\in [0,T]\}$ be, respectively, the solution of \eqref{eqn:2.2} and of\eqref{eqn:2.3}  with Assumptions \ref{gt:2.1} and \ref{gt:2.2}. We further assume that $|\sigma(t,x)|\geq \sigma_0>0$ for all $(t,x)\in [0,T]\times \mathbb{R}.$  Then, for each $\alpha\geq 1$ and $t\in (0,T],$ 
	\begin{equation*}
	d_{TV}(X^{\alpha}_t,X_t) \leq  Ct^{-1/2}\alpha^{-1/2},
	\end{equation*}
	where $C$ is a constant depending only on $\{x_0,y_0,\sigma_0, \kappa,\gamma,K,L,\,M,T\}$.	
\end{corollary}
	\subsection{Approximation the velocity and rescaled velocity processes}
	\label{sec: velocity process}
	In this section, we establish rates of convergence in $L^p$-distances and in the total variation distance for the velocity and rescaled velocity processes. We will discuss the re-scaled velocity process first since in this case, our results are applicable to more general settings where both external forces and diffusion coefficients can be dependent on both $x$ and $t$, i.e. $g=g(t,x)$ and $\sigma=\sigma(t,x)$.	
	\subsubsection{The re-scaled velocity process}
From the second equation of \eqref{eq1} we  can see that the process $Y^{\alpha}_{\frac{t}{\alpha}}$ satisfies
	\begin{equation}
		\label{eqn:2.10}
		\begin{cases}
			Y^{\alpha}_{\frac{t}{\alpha}} =y_0-(\kappa+\gamma)\int_0^tY^\alpha_{\frac{s}{\alpha}}ds- \int_0^tg(\frac{s}{\alpha},X^\alpha_{\frac{s}{\alpha}})ds-\gamma \int_0^tE(Y^\alpha_{\frac{s}{\alpha}})ds+\alpha \int_0^t\sigma(\frac{s}{\alpha},X^\alpha_{\frac{s}{\alpha}})dW_{\frac{s}{\alpha}}\\
			X^\alpha_0 = x_0.
		\end{cases}
	\end{equation}	
We recall the definition of the re-scaled velocity process introduced in the Introduction 
$$
\tilde{Y}^{\alpha}_t = \dfrac{1}{\sqrt{\alpha}}Y^{\alpha}_{\frac{t}{\alpha}}.
$$ 
Then $\tilde{Y}^{\alpha}_t $ satisfies \eqref{eqn:2.10m10}, that is
\begin{equation}
	\label{eqn:2.10m}
	\begin{cases}
	\tilde{Y}^{\alpha}_t =\dfrac{y_0}{\sqrt{\alpha}}-(\kappa+\gamma)\int_0^t\tilde{Y}^{\alpha}_sds- \dfrac{1}{\sqrt{\alpha}}\int_0^tg(\frac{s}{\alpha},X^\alpha_{\frac{s}{\alpha}})ds-\gamma \int_0^tE(\tilde{Y}^{\alpha}_{s})ds+\sqrt{\alpha} \int_0^t\sigma(\frac{s}{\alpha},X^\alpha_{\frac{s}{\alpha}})dW_{\frac{s}{\alpha}}\\
	X^\alpha(0) = x_0.
	\end{cases}
	\end{equation}
 Now, we put $\tilde{W}_t=\sqrt{\alpha}W_{t/\alpha}$, then $(\tilde{W}_t)_{t\geq 0}$ is a Brownian motion process and \eqref{eqn:2.10m} can be rewritten in the form 	
\begin{equation}
 \label{eqn:2.10m1}
 \begin{cases}
 \tilde{Y}^{\alpha}_t =\dfrac{y_0}{\sqrt{\alpha}}-(\kappa+\gamma)\int_0^t\tilde{Y}^{\alpha}_sds- \dfrac{1}{\sqrt{\alpha}}\int_0^tg(\frac{s}{\alpha},X^\alpha_{\frac{s}{\alpha}})ds-\gamma \int_0^tE(\tilde{Y}^{\alpha}_{s})ds+ \int_0^t\sigma(\frac{s}{\alpha},X^\alpha_{\frac{s}{\alpha}})d\tilde{W}_s\\
 X^\alpha_0 = x_0.
 \end{cases}
 \end{equation}
Our goal in this section is to study the rate of convergence in $L^p$-distance and in the total variation distance  between $\tilde{Y}^{\alpha}_t$ and $\tilde{Y}_t$. Here, $\tilde{Y}_t$ is the solution of Ornstein-Uhlembeck process \eqref{eqn:2.14m0}, which is
 \begin{equation}
 \label{eqn:2.14m}
 \begin{cases}
 d\tilde{Y}_t = -(\kappa+\gamma)d\tilde{Y}_t+\sigma(0,x_0)d\tilde{W}_t,\\
 \tilde{Y}(0) = 0.
 \end{cases}
 \end{equation}
First,  we obtain the rate of convergence in $L^p$-distances  between $\tilde{Y}^{\alpha}_t$ and $\tilde{Y}_t$ in the following lemma.
\begin{theorem}
	\label{bd:3.1}
	Let  $\{\tilde{Y}^\alpha_t, t\in [0,T]\}$ and $\{\tilde{Y}_t, t\in [0,T]\}$  be, respectively, the solution of \eqref{eqn:2.10m1} and of \eqref{eqn:2.14m} with Assumptions \ref{gt:2.1} and \ref{gt:2.3}. Then, for all $p \geq 2$ and $\alpha\geq 1$,
	\begin{equation*}
\E\left[	\sup_{0\leq t\leq T}|\tilde{Y}^\alpha_t-\tilde{Y}_t|^{p}\right] \leq \dfrac{C}{\alpha^{p/2}},
	\end{equation*}
	where $C$ is a positive constant depending on $p$  but not on $\alpha$.
\end{theorem}
\begin{proof}
	From \eqref{eqn:2.10m1} and \eqref{eqn:2.14m}, together with  the fact that $\E[\tilde{Y}_t]=0$ for all $t\in [0,T]$, we get
	\begin{align*}
	\tilde{Y}^\alpha_t-\tilde{Y}_t & =\dfrac{y_0}{\sqrt{\alpha}}-(\kappa+\gamma)\int_0^t(\tilde{Y}^{\alpha}_s-\tilde{Y}_s)ds- \dfrac{1}{\sqrt{\alpha}}\int_0^tg(\frac{s}{\alpha},X^\alpha_{\frac{s}{\alpha}})ds \\
	&\qquad -\gamma \int_0^tE(\tilde{Y}^{\alpha}_{s}-\tilde{Y}_s)ds + \int_0^t(\sigma(\frac{s}{\alpha},X^\alpha_{\frac{s}{\alpha}})-\sigma(0,x_0))d\tilde{W}_s.
	\end{align*}
	Using Minkowski's inequality \eqref{eq: Minkowski inequality} with $n=5$, H\"older's inequality, the BDG inequality and Assumptions \ref{gt:2.1} and \ref{gt:2.3}, one can derive that
	\begin{align*}
	\E\left[\sup_{0\leq s\leq t}|\tilde{Y}^\alpha_s-\tilde{Y}_s|^{p}\right] & \leq 5^{p-1}\Big[\dfrac{|y_0|^p}{\alpha^{p/2}}+ t^{p-1}(\kappa+\gamma)^p\int_{0}^{t}\E\big[|\tilde{Y}^{\alpha}_s-\tilde{Y}_s|^p \big]\,ds 
\\&\qquad\quad+ \dfrac{K^p(2t)^{p-1}}{\alpha^{p/2}} \int_0^t(1+\E\big[|X^\alpha_{\frac{s}{\alpha}}|^p\big])\,ds+t^{p-1}\gamma^p\int_{0}^{t}\E\big[|\tilde{Y}^{\alpha}_s-\tilde{Y}_s|^p\big]\,ds
\\&\qquad\quad+ 2^{p-1}C_pt^{p/2-1}\int_{0}^t(\dfrac{s^p}{\alpha^p}+\E\big[|X^\alpha_{\frac{s}{\alpha}}-x_0|^p)\big]\,ds\Big].
	\end{align*}
By Lemma \ref{bd:2.4}, with noting that $0\leq \frac{s}{\alpha}\leq s\leq t\leq T$, we have
\begin{align*}
\E\left[\sup_{0\leq s\leq t}|\tilde{Y}^\alpha_s-\tilde{Y}_s|^{p}\right] &\leq  \dfrac{C}{\alpha^{p/2}} +C\int_0^t \E\left[|\tilde{Y}^{\alpha}_s-\tilde{Y}_s|^p\right]\,ds+ C\int_{0}^t(\dfrac{s^p}{\alpha^p}+\dfrac{s^{p/2}}{\alpha^{p/2}})ds
\\&\leq  \dfrac{C}{\alpha^{p/2}} +C\int_0^t \E\left[\sup_{0\leq u\leq s}|\tilde{Y}^\alpha_u-\tilde{Y}_u|^{p}\right]\, ds,
\end{align*}	
	where $C$ is constant depending only on $\{x_0,y_0,\kappa,\gamma,K,L,p\}$. Using Growwall's inequality, we obtain the claimed inequality and complete the proof.
\end{proof}
From \eqref{eqn:2.10m} and \eqref{eqn:2.14m}, under Assumptions \ref{gt:2.1} the  Malliavin differentiability of
the solutions $\tilde{Y}^\alpha_t$ and $\tilde{Y}_t$
 can be proved by using the same method as in the proof of Lemma \ref{bd:2.3}. Moreover, the Malliavin derivatives $D_\theta\tilde{Y}^\alpha_t$ and $D_\theta\tilde{Y}_t$ satisfy $D_r\tilde{Y}^{\alpha}_t= D_r\tilde{Y}_t = 0$ for $r \geq t/\alpha$ and $0\leq \alpha r < t\leq T,$	\begin{align}\label{eqn:2.4.m}
	D_r\tilde{Y}^{\alpha}_t &=\sqrt{\alpha}\sigma(r,X^\alpha_r)-(\kappa+\gamma)\int_{\alpha r}^tD_r\tilde{Y}^{\alpha}_sds- \dfrac{1}{\sqrt{\alpha}}\int_{\alpha r}^t\bar{g}(\frac{s}{\alpha})D_rX^\alpha_{\frac{s}{\alpha}}ds+\sqrt{\alpha} \int_{\alpha r}^t\bar{\sigma}(\frac{s}{\alpha})D_rX^\alpha_{\frac{s}{\alpha}}dW_{\frac{s}{\alpha}},\notag\\
D_r\tilde{Y}_t &= \sqrt{\alpha}\sigma(0,x_0)-(\kappa+\gamma)\int_{\alpha r}^tD_r\tilde{Y}_sds,
	\end{align}
	where $\bar{g}(s), \bar{g}^\alpha(s),\bar{\sigma}(s), \bar{\sigma}^\alpha(s)$ are adapted stochastic processes and bounded by $L.$ Furthermore, 	if $g$ and $ \sigma$ are continuously differentiable, then $\bar{g}(s)=g'_2(s,X_s)$, $\bar{g}^\alpha(s)= g'_2(s,X_s^\alpha), \bar{\sigma}(s)= \sigma'_2(s,X_s)$ and $ \bar{\sigma}^\alpha(s)= \sigma'_2(s,X_s^\alpha).$
	\begin{lemma}
	\label{bd:2.13}
Let $\{\tilde{Y}^{\alpha}_t, t\in[0,T]\}$ and  $\{\tilde{Y}_t, t\in[0,T]\}$ be defined as above. Assume that Assumptions \ref{gt:2.1} and \ref{gt:2.3} hold. Then, for all  $\alpha\geq 1$,
	\begin{equation*}
	\E\left[\|D\tilde{Y}^{\alpha}_t-D\tilde{Y}_t\|_{\mathcal{H}}^2\right] \leq C\alpha^{-1},
	\end{equation*}
	where $C$ is constant depending only on $\{x_0,y_0,\kappa,\gamma,K,L,T\}$.
\end{lemma}
\begin{proof}
It follows from \eqref{eqn:2.4.m} that, for  $0\leq \alpha r < t,$ 
	\begin{align*}
D\tilde{Y}^{\alpha}_t-D\tilde{Y}_t &=\sqrt{\alpha}\left(\sigma(r,X^\alpha_r)-\sigma(0,x_0)\right) -(\kappa+\gamma)\int_{\alpha r}^{t} \left[D\tilde{Y}^{\alpha}_s-D\tilde{Y}_s\right]ds\\&\qquad -\dfrac{1}{\sqrt{\alpha}}\int_{\alpha r}^t\bar{g}(\frac{s}{\alpha})D_rX^\alpha_{\frac{s}{\alpha}}ds+\sqrt{\alpha} \int_{\alpha r}^t\bar{\sigma}(\frac{s}{\alpha})D_rX^\alpha_{\frac{s}{\alpha}}dW_{\frac{s}{\alpha}}.
	\end{align*}
Using Cauchy-Schwarz inequality, the  It\^o isometry formula,  Assumptions \ref{gt:2.1} and \ref{gt:2.3}, Lemma \ref{bd:2.4}, with noting that $0\leq \frac{s}{\alpha}\leq s\leq t\leq T$, we get 
	\begin{align*}
	\E\left[|D_r\tilde{Y}^{\alpha}_t-D_r\tilde{Y}_t|^2\right] &\leq 4\Big(2\alpha\left(r^2+\E\left[|X^\alpha_r-x_0|^2\right]\right) +(\kappa+\gamma)^2(t-\alpha r)\int_{\alpha r}^{t} \E\Big[\big|D_r\tilde{Y}^{\alpha}_s-D_r\tilde{Y}_s\big|^2\Big]\,ds\\
	&\qquad +\dfrac{L^2T}{\alpha}\int_{\alpha r}^t\E\big[|D_rX^\alpha_{\frac{s}{\alpha}}|^2\big]\,ds+L^2\alpha \int_{\alpha r}^t\E\big[|D_rX^\alpha_{\frac{s}{\alpha}}|^2\big]\frac{ds}{\alpha}\Big)\\&\leq C\Big(\alpha\left(r^2+r\right) +\int_{\alpha r}^{t} \E\Big[\big|D_r\tilde{Y}^{\alpha}_s-D_r\tilde{Y}_s\big|^2\Big]\,ds+\dfrac{(t-\alpha r)}{\alpha}+(t-\alpha r)\Big),
	\end{align*}
	where $C$ is constant depending only on $\{x_0,y_0,\kappa,\gamma,K,L,T\}$. From this, we have
\begin{align*}
&\int_{0}^{t/\alpha}\E\big[|D_r\tilde{Y}^{\alpha}_t-D_r\tilde{Y}_t|^2\big]\,dr
\\&\qquad\leq C\Big(\alpha\left(\frac{t^3}{3\alpha^3}+\frac{t^2}{2\alpha^2}\right) +	\int_{0}^{t/\alpha}\left(\int_{\alpha r}^{t} \E\Big[\big|D_r\tilde{Y}^{\alpha}_s-D_r\tilde{Y}_s\big|^2\Big]\,ds\right)dr+\dfrac{t^2}{2\alpha^2}+\dfrac{t^2}{2\alpha}\Big)\\
\\&\qquad\leq \dfrac{C}{\alpha} +	C\int_{0}^{t}\left(\int_{0}^{s/\alpha} E\left|D_r\tilde{Y}^{\alpha}_s-D_r\tilde{Y}_s\right|^2dr\right)ds.
\end{align*}
	Denote $\phi(t) = \int_{0}^{t/\alpha}E|D_r\tilde{Y}^{\alpha}_t-D_r\tilde{Y}_t|^2dr$, 
	using Gronwall's inequality, one sees easily that
	$$\E\Big[\|D\tilde{Y}^{\alpha}_t-D\tilde{Y}_t\|_{\mathcal{H}}^2\Big]=\phi(t)\leq \dfrac{C}{\alpha} e^{Ct} \leq C\alpha^{-1},$$
		where $C$ is a constant depending only on $\{x_0,y_0,\kappa,\gamma,K,L,T\}$. This completes our proof. 
\end{proof}
Bringing the above lemmas together, we can get the following result.
	\begin{theorem}
	\label{thm: TV for velocity}
	Let $\{\tilde{Y}^{\alpha}_t, t\in[0,T]\}$ and  $\{\tilde{Y}_t, t\in[0,T]\}$ be as before. Assume that Assumptions \ref{gt:2.1} and \ref{gt:2.3} hold and $|\sigma(0,x_0)|>0$ for all $(t,x)\in [0,T]\times \mathbb{R}.$ Then, for each $t\in (0,T],$ 
	\begin{equation*}
	d_{TV}(\tilde{Y}^{\alpha}_t,\tilde{Y}_t) \leq C\left(\lambda(t,2(\kappa+\gamma))\right)^{-1/2}\alpha^{-1/2},
	\end{equation*}
	where $C$ is a constant depending only on $\{x_0,y_0,\kappa,\gamma,K,L,T, \sigma(0,x_0)\}$.
\end{theorem}
\begin{proof}
Using Lemma \ref{bd:2.1} we have
\begin{equation*}
	d_{TV}(\tilde{Y}^{\alpha}_t,\tilde{Y}_t) \leq \|\tilde{Y}^{\alpha}_t-\tilde{Y}_t\|_{1,2}\left[3\left(\E\|D^2\tilde{Y}_t\|^4_{\mathcal{H}\bigotimes\mathcal{H}} \right)^{1/4}\left( \E\|D\tilde{Y}_t\|_\mathcal{H}^{-8}\right)^{1/4}+ 2\left(\E\|D\tilde{Y}_t\|_\mathcal{H}^{-2}\right)^{1/2}\right].
	\end{equation*}
	Thanks to Theorem \ref{bd:3.1} and Lemma \ref{bd:2.13}, we obtain
	\begin{equation}
	\label{eqn:3.4}
	d_{TV}(\tilde{Y}^{\alpha}_t,\tilde{Y}_t) \leq C\alpha^{-1}\left[3\left(\E\|D^2\tilde{Y}_t\|^4_{\mathcal{H}\bigotimes\mathcal{H}} \right)^{1/4}\left( \E\|D\tilde{Y}_t\|_\mathcal{H}^{-8}\right)^{1/4}+ 2\left(\E\|D\tilde{Y}_t\|_\mathcal{H}^{-2}\right)^{1/2}\right].
	\end{equation}
	Moreover, we have  $D_r\tilde{Y}_t = 0$ for $r \geq t/\alpha$ and $D_r\tilde{Y}_t = \sqrt{\alpha}\sigma(0,x_0)-(\kappa+\gamma)\int_{\alpha r}^tD_r\tilde{Y}_sds,$ for $0\leq \alpha r < t.$ Solving this ODE directly yields
	\begin{equation}\label{pt}
	D_r\tilde{Y}_t = \begin{cases}
	0 \quad &\text{if }r\alpha \geq t\\
	\sqrt{\alpha}\sigma(0,x_0)\exp{[(\kappa+\gamma)(r\alpha-t)]} \quad &\text{if }r\alpha < t.
	\end{cases}
	\end{equation}
	Thus, one can easily show that
	\begin{equation*}
	\|D\tilde{Y}_t\|^{2}_{\mathcal{H}} = \alpha\sigma^2(0,x_0)\int_{0}^{t/\alpha}\exp{[2(\kappa+\gamma)(r\alpha-t)]}dr = \sigma^2(0,x_0)\lambda(t,2(\kappa+\gamma)).
	\end{equation*}
	This implies, for all $p \geq 2$,  \begin{equation}\label{ptm1}\E\big[\|D\tilde{Y}_t\|^{-p}_{\mathcal{H}}\big] = \sigma^{-p}(0,x_0)\left(\lambda(t,2(\kappa+\gamma))\right)^{-p/2}.
		\end{equation}  
Applying the Malliavin derivative to \eqref{pt}, we have
	\begin{equation}
	\label{eqn:3.6}
	D_\theta D_r\tilde{Y}_t=0.
	\end{equation}
	
Combining \eqref{eqn:3.4},  \eqref{ptm1} and \eqref{eqn:3.6}, we obtain
	\begin{equation*}
	d_{TV}(\tilde{Y}^{\alpha}_t,\tilde{Y}_t) \leq C\left(\lambda(t,2(\kappa+\gamma))\right)^{-1/2}\alpha^{-1/2},
	\end{equation*}
	where $C$ is a constant depending only on $\{x_0,y_0,\kappa,\gamma,K,L,T, \sigma(0,x_0)\}$. This completes our proof. 
\end{proof}
\subsubsection{The velocity process}
As mentioned in the introduction, when $g(t,x)=g(x)$ and $\sigma(t,x)=\delta$, \cite[Theorem 2.3]{narita3} shows that the velocity process $Y^\alpha_t$ converges to the normal distribution as $\alpha \to \infty$.  In the rest of this section, we generalize this result to a much more general setting where $g$ depends on both $x$ and $t$ while $\sigma$ depends only on $t$, i.e. $\sigma(t,x)=\sigma(t)$.

From \eqref{eqmm1} we get
	\begin{equation}
	\label{eqn:2.10mm}
	Y^{\alpha}_t = y_0\exp{[-\alpha(\kappa+\gamma)t]}-\alpha(\kappa+\gamma)I_1^{\alpha}(t)+\alpha(\kappa+\gamma)I_2^{\alpha}(t)+\alpha(\kappa+\gamma)I_3^{\alpha}(t).
	\end{equation}
Since $X^{\alpha}_t$ is Malliavin differentiable, $Y^{\alpha}_t$ is also Malliavin  differentiable satisfying $D_rY^{\alpha}_t = 0$ for $r \geq t$,  and for $0\leq r < t\leq T$
	\begin{equation}
	\label{eqn:2.11}
	D_rY^{\alpha}_t = \alpha\sigma(r)\exp{[\alpha(\kappa+\gamma)(r-t)]}-\alpha\int_{r}^{t}\exp{[\alpha(\kappa+\gamma)(t-s)]}g'(s,X^{\alpha}_s)D_{r}X^{\alpha}_s\, ds. 
	\end{equation}
Define 
$$
W^{\alpha}(t) = \sqrt{\alpha}(\kappa+\gamma)I_3^{\alpha}(t)=\sqrt{\alpha}\int_{0}^{t} \exp{[\alpha(\kappa+\gamma)(u-t)]}\sigma(u) dW_u.
$$ 
Then $W^{\alpha}(t)$ is also Malliavin differentiable and $D_rW^{\alpha}_t = 0$ for $r \geq t$ and for $0\leq r < t\leq T$
	\begin{equation}
	\label{eqn:2.12}
	D_rW^{\alpha}(t) = \sqrt{\alpha}\sigma(r)\exp{[\alpha(\kappa+\gamma)(r-t)]}.
	\end{equation}
	\begin{lemma}
		\label{bd:2.9}
		Let $Y^\alpha_t$ be the solution of \eqref{eqn:2.10} with Assumptions \ref{gt:2.1}. Then, for all $p\geq 2$ and  $t\in (0,T],$
		\begin{equation*}
			\E\left[\Big|\dfrac{Y^{\alpha}_t}{\sqrt{\alpha}}-W^{\alpha}(t)\Big|^p\right] \leq C\alpha^{-p/2},
		\end{equation*}
		where $C$ is constant depending only on $\{y_0,\kappa,\gamma,K,L,p\}$.
	\end{lemma}
	\begin{proof}
From \eqref{eqn:2.10mm}, we get $$\dfrac{Y^{\alpha}_t}{\sqrt{\alpha}}-W^{\alpha}(t)=\dfrac{y_0}{\sqrt{\alpha}}\exp{[-\alpha(\kappa+\gamma)t]}-\sqrt{\alpha}(\kappa+\gamma)I_1^{\alpha}(t)+\sqrt{\alpha}(\kappa+\gamma)I_2^{\alpha}(t).$$ 
Then, by \eqref{eq1m}, \eqref{eq2m} and Lemma \ref{bd:2.6},  we have the following estimation
		\begin{equation*}
			\begin{split}
				\E\left[\Big|\dfrac{Y^{\alpha}_t}{\sqrt{\alpha}}-W^{\alpha}(t)\Big|^p\right] 
				&\leq 3^{p-1}\left[\dfrac{|y_0|^p}{\alpha^{p/2}}+\alpha^{p/2}(\kappa+\gamma)^p\E\big(|I_1^{\alpha}(t)|^p\big)+\alpha^{p/2}(\kappa+\gamma)^p|I_2^{\alpha}(t)|^p\right]\\
				&\leq C\left[\dfrac{1}{\alpha^{p/2}}+\alpha^{p/2}\left(2+\E\Big(\sup_{0\leq t\leq T}|X^{\alpha}_t|^{p}\Big)\right)(\lambda(t;(\kappa+\gamma)\alpha))^{p}+\alpha^{p/2}\left( \lambda(t,\alpha\kappa)\right)^p\right]\\
				&\leq \dfrac{C}{\alpha^{p/2}}\left[1+\E\Big(\sup_{0\leq t\leq T}|X^{\alpha}_t|^{p}\Big)\right]
				\\&\leq \dfrac{C}{\alpha^{p/2}},
			\end{split}
		\end{equation*}
		where $C$ is a constant depending only on $\{y_0,\kappa,\gamma,K, L\}$. This completes the proof of the lemma. 
	\end{proof}
	\begin{lemma}
		\label{bd:2.10}
		Let $Y^\alpha_t$ be the solution of \eqref{eqn:2.10} with Assumptions \ref{gt:2.1}. Assume that $\sigma(t)$ is continuous on $[0,T]$. Then, for all $\alpha\geq 1$,
		\begin{equation*}
		\E\int_{0}^{t}\left|\dfrac{D_rY^{\alpha}_t}{\sqrt{\alpha}}-D_rW^{\alpha}(t)\right|^pdr\leq C\|\sigma\|^p_\infty \alpha^{-p/2},
		\end{equation*}
		where $\|\sigma\|_\infty =\sup\limits_{t\in[0,T]}|\sigma(t)|$ and $C$ is constant depending only on $\{\kappa,\gamma,K,L,p,T\}$.
	\end{lemma}
	
	\begin{proof}
From \eqref{eqn:2.11} and \eqref{eqn:2.12} we have
		\begin{align}\label{eq3m}
				\E\int_{0}^{t}\left|\dfrac{D_rY^{\alpha}_t}{\sqrt{\alpha}}-D_rW^{\alpha}(t)\right|^pdr
				&= \E\int_{0}^{t}\left|\sqrt{\alpha}\int_{r}^{t}\exp{[\alpha(\kappa+\gamma)(t-s)]}g'(s,X^{\alpha}_s)D_{r}X^{\alpha}_sds\right|^p\,dr\notag\\
				&\leq C\alpha^{p/2}\int_{0}^{t}\E\left|\int_{r}^{t}\exp{[\alpha(\kappa+\gamma)(t-s)]}|D_{r}X^{\alpha}_s|ds\right|^pdr.
		\end{align}
On the other hand,	from 	\eqref{eqn:2.4}, together with Assumptions \ref{gt:2.1} and the fact that $\bar{\sigma}^\alpha(s)=0,$  we have
	\begin{align*}	|D_rX^{\alpha}_t| &\leq \dfrac{|\sigma(r)|}{\kappa+\gamma}(1-\exp{[\alpha(\kappa+\gamma)(r-t)]})+\dfrac{1}{\kappa+\gamma}\int_r^t|\bar{g}^\alpha(s)| |D_rX_s^\alpha| ds\notag\\
		&\qquad+\dfrac{1}{\kappa+\gamma}\int_{r}^{t}\exp{[\alpha(\kappa+\gamma)(s-t)]}|\bar{b}^\alpha(s)||D_rX^{\alpha}_s| ds\\&
		\leq \dfrac{\|\sigma\|_\infty}{\kappa+\gamma}+\dfrac{2L}{\kappa+\gamma}\int_r^t |D_rX_s^\alpha|\, ds.
		\end{align*}
		Thus, by Gronwall's inequality one sees that 
		\begin{equation*}
			|D_rX^{\alpha}_t|\leq  \dfrac{\|\sigma\|_\infty}{\kappa+\gamma} e^{\frac{2L(t-r)}{\kappa+\gamma}}\leq C\|\sigma\|_\infty, 
		\end{equation*}
		where $C$ is constant depending only on $\{\kappa,\gamma,K,L,T\}$.  Substituting this into \eqref{eq3m} yields
			\begin{align*}		\E\int_{0}^{t}\left|\dfrac{D_rY^{\alpha}_t}{\sqrt{\alpha}}-D_rW^{\alpha}(t)\right|^pdr
		&\leq C\|\sigma\|^p_\infty \alpha^{p/2}\int_{0}^{t}\left|\int_{r}^{t}\exp{[\alpha(\kappa+\gamma)(t-s)]}ds\right|^pdr\\&\leq C\|\sigma\|^p_\infty \alpha^{p/2}\int_{0}^{t}\lambda^p(t-r;\alpha(\kappa+\gamma))dr
	\\&\leq C\|\sigma\|^p_\infty\alpha^{-p/2},
		\end{align*}
		which is the desired conclusion.
	\end{proof}
	
	Now we are ready to prove the rate of convergence in total variation distance for the velocity process $Y^\alpha_t$ as $\alpha \to \infty$.	
	\begin{theorem}
	\label{thm: TV for velocity2}
	Let $Y^\alpha_t$ be the solution of \eqref{eqn:2.10} with   Assumptions \ref{gt:2.1}. Assume that $\sigma(t)$ is a continuously differentiable function  on $[0,T]$ and that $\sigma(t)\not=0$ for each $t \in (0, T]$. Then, for each $\alpha\geq 1$ and $t\in (0,T]$ 
		\begin{equation*}
			d_{TV}\left(\dfrac{Y^{\alpha}_t}{\sqrt{\alpha}},N\right) \leq C\left(\lambda(t,2(\kappa+\gamma))\right)^{-1/2}\alpha^{-1/2},
		\end{equation*}
		where $N$ is a normal random variable with mean $0$ and variance $\dfrac{
		\sigma^2(t)}{2(\kappa+\gamma)}$, and $C$ is a constant depending only on $\{x_0,y_0,\kappa,\gamma,K,L,T, \sigma\}$.
	\end{theorem}
	\begin{proof}
		Using Lemma \ref{bd:2.1},  we have:
\begin{align*}
d_{TV}\left(\dfrac{Y^{\alpha}_t}{\sqrt{\alpha}},W^{\alpha}(t)\right) \leq
			\left\|\dfrac{Y^{\alpha}_t}{\sqrt{\alpha}}-W^{\alpha}(t)\right\|_{1,2}
		&\Bigg[3\left(\E\|D^2W^{\alpha}(t)\|^4_{\mathcal{H}\bigotimes\mathcal{H}} \right)^{1/4}\left( \E\|DW^{\alpha}(t)\|_\mathcal{H}^{-8}\right)^{1/4}\\& \qquad \qquad + 2\left(\E\|DW^{\alpha}(t)\|_\mathcal{H}^{-2}\right)^{1/2}\Bigg].
		\end{align*}
		By  Lemmas  \ref{bd:2.9} and \ref{bd:2.10}, one can derive that
		\begin{align}
			\label{eqn:3.7}
		d_{TV}\left(\dfrac{Y^{\alpha}_t}{\sqrt{\alpha}},W^{\alpha}(t)\right)  \leq C\alpha^{-1/2}&\Bigg[3\left(\E\|D^2W^{\alpha}(t)\|^4_{\mathcal{H}\bigotimes\mathcal{H}} \right)^{1/4}\left( \E\|DW^{\alpha}(t)\|_\mathcal{H}^{-8}\right)^{1/4}\notag\\& \ \ \qquad  + 2\left(\E\|DW^{\alpha}(t)\|_\mathcal{H}^{-2}\right)^{1/2}\Bigg],
		\end{align}
		where $C$ is a constant depending only on $\{x_0,y_0,\kappa,\gamma,K,L,T\}$.
		
		Now we calculate the derivatives of $W^{\alpha}(t)$. One can easily show that \begin{equation}\label{eqn:3.8}D_rW^{\alpha}(t) = \sqrt{\alpha}\sigma(r)\exp{[\alpha(\kappa+\gamma)(r-t)]} \ \mbox{and} \ \  
			D_\theta D_r W^{\alpha}(t) =0.
			\end{equation}
Therefore,
			\begin{align*}	
		\|DW^{\alpha}(t)\|_\mathcal{H}^2 &= \int_{0}^{t}\sigma^2(r)\alpha\exp{[2\alpha(\kappa+\gamma)(r-t)]}dr\\& \geq \alpha \lambda(t;2\alpha(\kappa+\gamma))\min\limits_{t\in [0,T]}|\sigma(t)|^2
		\\&\geq \lambda(t;2(\kappa+\gamma))\min\limits_{t\in [0,T]}|\sigma(t)|^2.
		\end{align*}
		Thus, for all $p \geq 2$,   
		\begin{equation}
			\label{eqn:3.9}
			\E\big[\|D\tilde{Y}_t\|^{-p}_{\mathcal{H}}\big]\leq \dfrac{1}{\lambda^{p/2}(t;2(\kappa+\gamma))\min\limits_{t\in [0,T]}|\sigma(t)|^p}.
		\end{equation}
		
		From \eqref{eqn:3.7}, \eqref{eqn:3.8}, \eqref{eqn:3.9}, we obtain
		\begin{equation*}
			d_{TV}\left(\dfrac{Y^{\alpha}_t}{\sqrt{\alpha}},W^{\alpha}(t)\right)\leq C\alpha^{-1/2}\left(\dfrac{2}{(\lambda(t;2(\kappa+\gamma)))^{1/2}\min\limits_{t\in [0,T]}|\sigma(t)|}\right),
		\end{equation*}
		where $C$ is a constant depending only on $\{x_0,y_0,\kappa,\gamma,K,L,T, \sigma\}$.
	 Thus,
		\begin{equation*}
			d_{TV}\left(\dfrac{Y^{\alpha}_t}{\sqrt{\alpha}},W^{\alpha}(t)\right) \leq C\left(\lambda(t,2(\kappa+\gamma))\right)^{-1/2}\alpha^{-1/2},
		\end{equation*}
	where $C$ is a constant depending only on $\{x_0,y_0,\kappa,\gamma,K,L,T, \sigma\}$.
	
Note that by It\^{o}'s isometry and using integration by parts for the non-stochastic integral, we have
	\begin{align*}
	\E\big[\left[W^{\alpha}(t)\right]^2\big]& = \int_{0}^{t}\alpha\sigma^2(r)\exp{[2\alpha(\kappa+\gamma)(r-t)]}dr\\&
	=\dfrac{\sigma^2(t)}{2(\kappa+\gamma)}-\dfrac{\sigma^2(0)\exp{[-2\alpha(\kappa+\gamma)t]}}{2(\kappa+\gamma)}-\int_{0}^{t}\sigma(r)\sigma'(r)\exp{[2\alpha(\kappa+\gamma)(r-t)]}dr.
	\end{align*}
	Thus, we can deduce that $W^{\alpha}(t)$ is random variable with normal distribution with mean $0$ and variance $$\dfrac{\sigma^2(t)}{2(\kappa+\gamma)}-\dfrac{\sigma^2(0)\exp{[-2\alpha(\kappa+\gamma)t]}}{2(\kappa+\gamma)}-\int_{0}^{t}\sigma(r)\sigma'(r)\exp{[2\alpha(\kappa+\gamma)(r-t)]}dr.$$
	
	Now, applying Lemma 4.9, \cite{klartagdtvnorm},	 we derive that
		\begin{align*}
			d_{TV}(W^{\alpha}(t),N)& \leq C\left(\dfrac{\sigma^2(0)\exp{[-2\alpha(\kappa+\gamma)t]}}{2(\kappa+\gamma)}+\int_{0}^{t}|\sigma(r)\sigma'(r)|\exp{[2\alpha(\kappa+\gamma)(r-t)]}dr\right)\\& 
			\leq C\left(\dfrac{\sigma^2(0)\exp{[-2\alpha(\kappa+\gamma)t]}}{2(\kappa+\gamma)}+\|\sigma\|_\infty \|\sigma'\|_\infty\lambda(t,2\alpha(\kappa+\gamma))\right)
			\\&\leq C\|\sigma\|_\infty \|\sigma'\|_\infty\alpha^{-1},
\end{align*}
		where $\|\sigma\|_\infty =\sup\limits_{t\in[0,T]}|\sigma(t)|$, $\|\sigma'\|_\infty =\sup\limits_{t\in[0,T]}|\sigma'(t)|$ and $C$ is an universal constant. Thus, 
		\begin{equation*}
			d_{TV}\left(\dfrac{Y^{\alpha}_t}{\sqrt{\alpha}},N\right) \leq d_{TV}(W^{\alpha}(t),N)+ d_{TV}\left(\dfrac{Y^{\alpha}_t}{\sqrt{\alpha}},W^{\alpha}(t)\right)\leq C\left(\lambda(t,2(\kappa+\gamma))\right)^{-1/2}\alpha^{-1/2},
		\end{equation*}
		where  $C$ is a constant depending on $\{x_0,y_0,\kappa,\gamma,K,L,T, \sigma\}$. This completes the proof. 
	\end{proof}
 
\section*{Acknowledgment} Research of MHD was supported by EPSRC Grants EP/W008041/1 and  EP/V038516/1.
\bibliographystyle{alpha}
\bibliography{refs}
\end{document}